\documentclass[10pt]{amsart}
\usepackage[T1]{fontenc}
\usepackage[english]{babel}
\usepackage{amssymb}
\usepackage{mathtools}
\usepackage{amsmath, empheq}
\usepackage[latin1]{inputenc} 
\usepackage[margin=1.4in]{geometry}
\usepackage{amsthm}
\usepackage{cite}
\usepackage{enumerate}
\usepackage{verbatim}
\usepackage[colorlinks=true,linkcolor = black,
  citecolor  = black , urlcolor=black, pdfborder={0 0 0}]{hyperref}
\newcommand{\partn}[1]{{\smallskip \noindent \textbf{#1.}}}

\DeclareMathOperator{\ord}{ord}

\DeclareMathOperator{\wideg}{wideg}

\DeclareMathOperator{\Z}{\mathbb{Z}}
\DeclareMathOperator{\F}{\mathbb{F}}

\DeclareMathOperator{\Q}{\mathbb{Q}}
\DeclareMathOperator{\Oc}{\mathcal{O}}

\numberwithin{equation}{section}

\begin{document}

\renewcommand{\thefootnote}{\roman{footnote}} 
\newtheorem{mydef}{Definition} %initiate defintion
\newtheorem{thm}{Theorem} %initiatie thm
\newtheorem{MYthm}{Theorem}
\renewcommand*{\theMYthm}{\Alph{MYthm}}
\newtheorem{example}{Example} %initiate example
\newtheorem{conj}{Conjecture}
\newtheorem{prop}{Proposition}
\newtheorem{corr}{Corollary} %corollary
\newtheorem{MYcorr}{Corollary}
\renewcommand*{\theMYcorr}{\Alph{MYcorr}}
\newtheorem{lemma}{Lemma} %corollary
\newtheorem{rem}{Remark}
\newtheorem*{prob}{Problem}
\newtheoremstyle{named}{}{}{\itshape}{}{\bfseries}{.}{.5em}{\thmnote{#3}}
\theoremstyle{named}
\newtheorem*{namedlemma}{Lemma}

\pagenumbering{roman}

\title{Geometric location of periodic points of 2-ramified power series}
\author{Karl-Olof Lindahl and Jonas Nordqvist}
%\email{jonas.nordqvist@lnu.se}
%\address{Department of Mathematics, Linn\ae us University, V\"{a}xj\"{o}, Sweden}

\begin{abstract}
In this paper we study the geometric location of periodic points of power series defined over fields of positive characteristic $p$.  We find a lower bound for the norm of all nonzero periodic points in the open unit disk of 2-ramified power series. We prove that this bound is optimal for a large class of power series. Our main technical result is a computation of  the first significant terms of $p$-power iterates of 2-ramified power series.  As a by-product we obtain a self-contained proof of the characterization of 2-ramified power series.

%We are able to give optimal conditions for a subset of 2-ramified power series for which we can explicitly compute the radius of the sphere on which the periodic points are found., by computing the first significant terms of the $p^n$th iterate of 2-ramified power series. 

\vspace{1ex}

\vspace{1ex}
\end{abstract}

\maketitle

\setcounter{page}{1}
\pagenumbering{arabic}

\section{Introduction}
%Arithmetic dynamical systems Function itertation is aDiscrete time dynamical systems over 
%Arithmetic dynamical systems defined by function iterar
%In this paper we study arithmetic dynamical systems defined by function iteration over  \cite{Silverman2007}. Periodic points are the first point of %study in the theory of discrete time dynamical systems generated by the iteration of functions. 
%Recall that a point $\zeta_0$ said to be a of minimal period $n$ of a analytic map $f$ in one variable is 

The study of periodic points is a central issue in the theory of dynamical systems.
In this article we are interested in the geometric location of periodic points of power series defined over fields of positive characteristic. 
Dynamics over fields of positive characteristic is an important topic in arithmetic dynamics \cite{Silverman2007,AnashinKhrennikov2009}. 
Lindahl and Rivera-Letelier \cite{Lindahl2013,LindahlRiveraLetelier2015, LindahlRiveraLetelier2013}  showed that there is a connection between the geometric location of periodic points of power series with integer coefficients, and lower ramification numbers of wildly ramified field automorphisms. %Lower ramification numbers of a power series $g$ satisfying $g(0)=0$ and $g'(0)=1$ are, up to a constant, the degree of the first non-linear term of $p$-power iterates of $g$. 
We utilize this connection to obtain an optimal lower bound for norms of periodic points of power series having a certain sequence of lower ramification numbers.

Throughout let $p$ be a prime and  $k$ be a field of characteristic $p$. Denote by $\ord(\cdot)$ the valuation on $k[[\zeta]]$ defined for a nonzero power series as the lowest degree of its nonzero terms, and put $\ord(0) = +\infty.$ Let $g \in k[[\zeta]]$ be a power series satisfying $g(0)=0$ and $g'(0) = 1$. Then for all integers $n\geq0$ we define the corresponding \emph{lower ramification number} of $g$ as 
\[i_n(g) := \ord\left(\frac{g^{p^n}(\zeta)-\zeta}{\zeta}\right).\] 
A famous theorem by Sen \cite{Sen1969,Lubin1995,LindahlRiveraLetelier2015} states that for an integer $n\geq1$ if $i_n(g) < + \infty$, then $i_n(g) \equiv i_{n-1}(g) \pmod{p^n}$. Thus given $i_0(g) \geq 1$ we have \begin{equation}\label{minramif}i_n(g) \geq 1 + p + \dots + p^n,\end{equation} for all $n\geq1$. If (\ref{minramif}) holds with equality we say that $g$ is \emph{minimally ramified}. Moreover, let $b\geq1$ be an integer, and suppose that for all integers $n\geq0$  we have \begin{equation}\label{def2ramif}i_n(g) = b(1+p+\dots +p^n).\end{equation} Then we say that $g$ is \emph{$b$-ramified}. 

Of particular interest in this paper are periodic points of power series that are $2$-ramified. 
Recall that the minimal period of each periodic point $\zeta_0 \neq0$ of $g$ in the open unit disk of $k$ is of the form $p^n$, for some integer 
$n\geq0$, see for example \cite[Lemma 2.1]{LindahlRiveraLetelier2013}. 

%One of our main results is a lower bound for the norm of periodic points   %of minimal period $p^n$, for $n\geq1$, 
%of 2-ramified power series. %For the case of minimally ramified power series this was done in \cite{LindahlRiveraLetelier2015}. 

%In the upcoming section we state and discuss the main results of this paper in more detail.

%Lindahl and Rivera-Letelier \cite{LindahlRiveraLetelier2015} 

%In \cite{Fransson2017}

%\section{Notes to self}
%All theorems are only valid for $p\geq 5$ since Lemma \ref{polydoublefact} is only valid for polynomials of degree strictly less than 3 when $p=3$, and degree three polynomials are used in some cases. Therefore we need a special proof for $p=3$ to obtain completeness.

%Also for $p=5$ there are some problems for Lemma 6 and 7. So those two cases needs to be taken care of separately.

%Solve the Lemma 5 usage for $p=3$.

\subsection{Main results} The main result of this paper is the lower bound for the norm of periodic points
%of minimal period $p^n$ for $n\geq1$ 
of 2-ramified power series given in Theorem \ref{perpoints} below.
 We also give sufficient conditions for optimality in this lower bound as discussed in \textsection\ref{optimal}.
%This result

\begin{MYthm}\label{perpoints}
Let $p\geq5$ be a prime and let $k:=(k,|\cdot|)$ be an ultrametric field of characteristic $p$. Let $f$ be a power series with coefficients in the closed unit disk of $k$, of the form \begin{equation*}\label{f}f(\zeta) = \zeta\left(1 + \sum_{j=2}^{+\infty}a_j\zeta^j\right).\end{equation*} Put $\lambda := a_2^{p-3}(3/2a_2^3 + a_3^2-a_2a_4)$. Let %$n\geq1$ be an integer and 
$\zeta_0$ be a periodic point of $f$ in the open unit disk of $k$. %be a periodic point of minimal period $p^n$. 
Then 
\[|\zeta_0| \geq |\lambda|^{\frac{1}{p}},\]
provided that $\zeta_0$ is not a fixed point. If $\zeta_0\neq 0$ is a fixed point of $f$ we have $|\zeta_0|\geq |a_2|$.
%Then for every integer $n\geq 1$ all periodic points $\zeta_0$ of $g$, of minimal period $p^n$  in the open unit disk of $k$, are located in \[\left\{\zeta \in k: |\zeta| \geq |\lambda|^{\frac{1}{p}}\right\}.\]
%\[S_{r}(0)\]% \[C_0 = \left\{\right\}
%\[|\zeta_0| \geq \left|a_2^{p-3}\varphi\right|^{\frac{1}{p}},\] with equality 
%if $|a_2|<1$, $|a_3|=1$, $p\geq5$ and $f$ is 2-ramified.
\end{MYthm}
\begin{rem}
We note that $\lambda \neq 0$ is equivalent to $f$ being 2-ramified.
\end{rem}
%\begin{rem}
%By \cite[Lemma 2.1]{LindahlRiveraLetelier2013} the minimal period of each periodic point $\zeta_0$ of $f$ in the open unit disk of $k$ is of the %form $p^n$, for some integer $n\geq0$. 
%\end{rem}
%\begin{rem}
%It follows immediately from \cite[Lemma 2.4]{LindahlRiveraLetelier2015} that for fixed points $\zeta_0\neq 0$ of $f$ in the open unit disk of $k$ %we have $|\zeta_0|\geq |a_2|$.
%\end{rem}
%\begin{rem}
%For fixed points of $f$ in the open unit disk of $k$ a similar result can be derived 
%\end{rem}
Theorem \ref{perpoints} is a consequence of Theorem \ref{pncoeff} and \cite[Lemma 2.4]{LindahlRiveraLetelier2015}. %The main work is a technical  computation of the first significant terms of iterates of the power series $f$.% at its $p^n$th iterate.% as manifested in Theorem \ref{pncoeff} below.
%\begin{rem}
%The condition $|a_3|=1$ is not redundant since it implies $\widetilde{f} \neq \id$.
%\end{rem}
%\begin{rem}
%\end{rem}
\begin{MYthm}\label{pncoeff}
Let $p$ be an odd prime and let $k$ be a field of characteristic $p$. Let $f \in k[[\zeta]]$ be defined as \[f(\zeta)  = \zeta\left(1 + \sum_{j=2}^{+\infty}a_j\zeta^j\right).\] Furthermore, let $n\geq1$ be an integer, and let \[d(n)=d := \frac{p^n-1}{p-1},\quad \varphi:=3/2a_2^3 + a_3^2-a_2a_4.\]
Let $\alpha_n, \beta_n$ and $\gamma_n$ be defined as follows 
\begin{equation*}\label{alphabetagmma2}
\alpha_n = a_2^{p^n-2d}\varphi^d, \quad
\beta_n = \frac{a_3}{a_2}\alpha_n, \quad
\gamma_n = -\left(\frac{3a_2}{2}-\frac{a_4}{a_2}\right)\alpha_n.
\end{equation*} Then
\[f^{p^n}(\zeta)-\zeta\equiv \alpha_n\zeta^{2d+1} + \beta_n\zeta^{2d+2} + \gamma_n\zeta^{2d+3} \mod \langle \zeta^{2d+4} \rangle.\]
\end{MYthm}
The main work of the proof is the Main Lemma given in \textsection\ref{secproof} where we compute the first significant terms of $f$ at its $p$th iterate by solving systems of difference equations in characteristic $p$. This result is an extension of Proposition 1 in \cite{Fransson2017}. We also note that due to Theorem \ref{pncoeff} we have a self-contained proof of Theorem 1 in \cite{Fransson2017}.

Below we discuss sufficient conditions for optimality of the bound in Theorem \ref{perpoints} in terms of the reduction of $f$.

\subsubsection{Optimality condition}\label{optimal}

Throughout the paper let $(k, |\cdot|)$ be an ultrametric field, $\mathcal{O}_k$ denote the ring of integers of $k$, and $\mathfrak{m}_k$ its maximal ideal. Geometrically, $\mathfrak{m}_k$ is the open unit disk in $k$. Let $\widetilde{k}:=\mathcal{O}_k/\mathfrak{m}_k$ be the residue field of $k$. Denote the projection in $\widetilde{k}$ of an element $a$ of $\mathcal{O}_k$ by $\widetilde{a}$; it is the reduction of $a$. The reduction of a power series $g \in \mathcal{O}_k[[\zeta]]$ is the power series $\widetilde{g}(\zeta) \in \widetilde{k}[[\zeta]]$ whose coefficients are the reductions of the corresponding coefficients of $g$.

The following result gives sufficient conditions for optimality of the bound in Theorem \ref{perpoints}.
\begin{MYcorr}\label{corrperpoints}
Let $p\geq5$ be a prime and let $k:=(k,|\cdot|)$ be an ultrametric field of characteristic $p$. Let $f \in \mathcal{O}_k[[\zeta]]$ be of the form 
\[f(\zeta)  = \zeta\left(1 + \sum_{j=2}^{+\infty}a_j\zeta^j\right).\] 
Put $\lambda := a_2^{p-3}(3/2a_2^3 + a_3^2-a_2a_4)$. 
Furthermore let $\lambda\neq0$, and $\widetilde{f}(\zeta)$ be 3-ramified. Then %for every integer $n\geq 1$ 
all periodic points of $f$ %, of minimal period $p^n$ with $n\geq 1$,  
in the open unit disk of $k$ that are not fixed points, are on the sphere 
\[\left\{\zeta \in k: |\zeta| = |\lambda|^{\frac{1}{p}}\right\}.\]
\end{MYcorr}
\begin{rem}\label{appremark}
By \cite[Corollary 1]{LaubieSaine1998} $\widetilde{f}(\zeta)$ is 3-ramified if $i_0(\widetilde{f}) =3$ and $i_1(\widetilde{f})=3+3p$. %Thus, by utilizing \cite[Main Lemma]{LindahlRiveraLetelier2015} we make the following specification for Corollary \ref{perpoints}. 
In particular, if $f\in k[[\zeta]]$ is of the form \[f(\zeta) \equiv \zeta(1 + a_2\zeta^2 + a_3\zeta^3) \mod \langle \zeta^{10} \rangle, \text{ $0<|a_2|<1$ and $|a_3|=1$.}\] Then $\lambda\neq0$ and $\widetilde{f}(\zeta)$ is 3-ramified. See Appendix \ref{apprem} for details. %Hence, the conditions in Corollary \ref{corrperpoints} are satisfied.
\end{rem}
The following example illustrates that the condition in Corollary \ref{corrperpoints} that $\widetilde{f}(\zeta)$ is 3-ramified is not redundant. 
See Appendix \ref{appexp} for details.
%\begin{prop}\label{perpoints3}

\begin{example}\label{appexample}
Let $p=5$, $k=\F_5((t))$ and $q(\zeta) \in \mathcal{O}_k[[\zeta]]$ be a power series of the form 
\[q(\zeta) \equiv \zeta(1 + c_2\zeta^2 + c_3\zeta^3 + c_4\zeta^4) \mod \langle \zeta^6 \rangle.\] 
Put $\lambda := c_2^{p-3} (3/2c_2^3 + c_3^2-c_2c_4)$. 
If $(c_2,c_3,c_4) = (1+t,1,0)$ then all periodic points $\zeta_0$ in $\mathfrak{m}_k$ of $g$, 
of minimal period $p^n$ with $n\geq 1$, 
have norm $|\zeta_0| = |\lambda|^{\frac{1}{p}}$.    
%\[\left\{\zeta \in k: |\zeta| = |\lambda|^{\frac{1}{p}}\right\},\] and 
If $(c_2,c_3,c_4) = (2+t,4,4)$, then $|\zeta_0|>|\lambda|^{\frac{1}{p}}$. In both cases $q$ is 2-ramified, but in none of the cases $\widetilde{q}$ is 3-ramified.
%t \[q_1(\zeta) = \zeta + (1+t)\zeta^3 + \zeta^4,\] and \[q_2(\zeta) = \zeta + (2+t)\zeta^3 + 4\zeta^4 + 4\zeta^5.\] Then both $q_1$ and $q_2$ are 2-ramified but \[i_n(\widetilde{q}_1) = 2 + 3p +3p^2+\dots+3p^n,\quad \text{ and } \quad i_n(\widetilde{q}_2) = 2 + 5p +5p^2+\dots+5p^n.\] Thus, the periodic points of $q_1$ are located \emph{on} the optimal sphere and for $q_2$ the inequality is strict.XXX
\end{example}
%\end{prop}
%The proofs of Theorem \ref{pncoeff} and Proposition \ref{prop2ramif} are found in ?\ref{secproof}.

%\subsection{Examples}
%Let $p\geq5$ be a prime $g\in \F_p[[t]][x]$ be of the form \[g(x) = x + tx^3+x^4+2x^7.\] Then by Theorem 1 in \cite{Fransson2017} $g$ is 2-ramified. However, by the Main Lemma of \cite{LindahlRiveraLetelier2015} we have that $\wideg(\widetilde{g})\geq $  

\subsection{Related works}
In \cite{LindahlRiveraLetelier2015} the authors give a corresponding result of Theorem \ref{perpoints} for minimally ramified power series, where $\lambda$ is expressed in terms of the coefficients of the first two non-linear  terms. 
Provided the information from Theorem \ref{pncoeff} we can make a corresponding version of Corollary \ref{corrperpoints} for minimally ramified power series, where the conditions for optimality are expressed in terms of the four lowest degree non-linear terms.

%In view of Corollary \ref{corrperpoints}, in case of optimality, the nonzero periodic points of $f$ in the open unit disk are concentrated on a sphere %about the origin. This is in contrast to the characteristic zero case where the periodic points are 

%where we explicitly can characterize sufficient conditions on the power series' first four coefficients such that the optimality conditions are met.
%\begin{MYcorr}\label{optminramif}
%Let $p$ be an odd prime and let $k:=(k,|\cdot|)$ be an ultrametric field of characteristic $p$. Let $f \in k[[\zeta]]$ be of the form \[f(\zeta)  = \zeta\left(1 + \sum_{j=1}^{+\infty}a_j\zeta^j\right).\] Put $\psi := a_1^{p-2}(a_1^2-a_2)$, and $\varphi := 3/2a_2^3 + a_3^2-a_2a_4$. Furthermore, let $0<|a_1| < 1$ and $|a_2|=1$, and also $\psi\neq0$ and $\widetilde{\varphi}\neq0$. Then for every integer $n\geq 1$ all periodic points of $f$, of minimal period $p^n$  in the open unit disk of $k$, are concentrated on the sphere \[\left\{\zeta \in k: |\zeta| = |\psi|^{\frac{1}{p}}\right\}.\]
%\end{MYcorr}

Let $\eta \in k$ and $f$ be a power series in $k[[\zeta]]$ of the form \[f(\zeta) = \eta\zeta +\cdots.\] In this paper we study the parabolic case where $\eta = 1$. The irrationally indifferent case, where the multiplier $\eta$ is of norm one but not a root of unity, was studied in \cite{LindahlRiveraLetelier2013}. The $p$-adic case was also studied in \cite{ArrowsmithVivaldi1994,Lindahl2013,Lubin1994,Rivera-Letelier2003}. In contrast to our case where the periodic points of period greater than one are concentrated on a single sphere inside the open unit disk, 
in the latter cases the periodic points are distributed on infinitely many different spheres. %inside the open unit disk.

The method used in this paper to calculate the coefficients of $p$-power iterates boils down to solving systems of difference equations over non-Archimedean fields. See for example \cite{MukhamedovAkin2015} for a recent contribution to this field of research.

%Lower ramification numbers are also studied in 
%\cite{Keating1992,LaubieSaine1997, %LaubieSaine1998,LaubieMovahhediSalinier2002,Lubin1995,Fransson2017,Sen1969,Wintenberger2004}.

%In \cite{LindahlRiveraLetelier2015} the authors study the dynamics near parabolic cycles of analytic maps in positive characteristic. In particular the authors are interested in the case

%study maps $f \in k[[\zeta]]$ such that $f(0)=0$ and $(f'(0))^q = 1$, where $q$ denote the order of $f'(0)$. Our work is mainly built on the results of this paper.

%The same authors 

%Similar results were found by the author in \cite{Lindahl2013} where...

%\begin{conj}
%Let $p$ be a prime and $k$ a field of characteristic $p$. Let $b\leq p-2$ be an integer and let $g \in k[[\zeta]]$ be a power series of the form \[g(\zeta) = \zeta\left(1 + \sum_{j=b}^{+\infty} a_j\zeta^j\right).\] Furthermore let $q(a_b,\dots,a_{2b})$ be the ramification unique polynomial of $g$, and put $\lambda := a_b^{p-b}q(a_b,\dots,a_{2b})$. Also let $g$ be $b$-ramified and $\widetilde{g}$ be $(b+1)$-ramified. Then for every integer $n\geq1$ all periodic points $\zeta_0$ of $g$, of minimal period $p^n$ in the open unit disk of $k$ are concentrated on the sphere \[\left\{\zeta\in k : |\zeta| = |\lambda|^{\frac{1}{p}}\right\}.\]
%\end{conj}

\subsection{Organization of the paper}
%The paper is structured as follows. 
Theorem \ref{perpoints} and Corollary \ref{corrperpoints} is proven in section \textsection\ref{proofmain} assuming Theorem \ref{pncoeff}. The main technical result of this paper is the computation of the first significant terms of $f$ at its $p$th iterate which is done in the Main Lemma in \textsection\ref{secproof}. This result is an extension of Proposition 1 in \cite{Fransson2017}. The setup for this proof is discussed in \textsection\ref{techres}, and we state and prove the Main Lemma in \textsection\ref{secproof} together with the proof of Theorem \ref{pncoeff}.

%\begin{proof}
%First we note that if $g$ is not $b$-ramified, then $\varphi=0$. Let $d=(p^n-1)/(p-1)$. Assuming Theorem \ref{pncoeff}, we have that \[g^{p^n}(\zeta) = a_2^{p^n-2d}\varphi^d\zeta^{2d +1} + \langle \zeta^{2d+2} \rangle.\] 
%\end{proof}
\section{Proof of Theorem \ref{perpoints} and Corollary \ref{corrperpoints} assuming Theorem \ref{pncoeff}}\label{proofmain}
In this section we prove Theorem \ref{perpoints} assuming Theorem \ref{pncoeff}. The proof is a direct consequence of Theorem \ref{pncoeff} and \cite[Lemma 2.4]{LindahlRiveraLetelier2015}. Before stating the special case of \cite[Lemma 2.4]{LindahlRiveraLetelier2015} utilized in the proof of Theorem \ref{perpoints} we give the following definitions.

%We state the necessary components of this lemma here, but first we state the following definition.

\begin{mydef}Let $p$ be a prime number and $k$ field of characteristic $p$. For a power series $g(\zeta)$ in $k[[\zeta]]$ of the form \[g(\zeta) = \zeta + \dots,\] define for each integer $n \geq 0$ the element $\mu_n(g)$ of $k$ as follows: Put $\mu_n(g) := 0$ if $i_n(g) = +\infty$, and otherwise let $\mu_n(g)$ be the coefficient of $\zeta^{i_n(g)+1}$ in the power series $g^{p^n}(\zeta)-\zeta$.
\end{mydef}

For a power series $f(\zeta)$ in $\Oc_k[[\zeta]]$, the Weierstrass degree $\wideg(f)$ of $f$ is the order in $\widetilde{k}[[z]]$ of the reduction $\widetilde{f}(z)$ of $f(\zeta)$. Note that if $\wideg(f)$ is finite, the number of zeros of $f$ in $\mathfrak{m}_k$, counted with multiplicity, is less than or equal to $\wideg(f)$; see e.g. \cite[\textsection VI, Theorem 9.2]{Lang2002}.

%Thus, if $g$ is 2-ramified then by Theorem \ref{pncoeff} we have $\mu_n(g) = \alpha_n$.
\begin{lemma}[Special case of Lemma 2.4 in \cite{LindahlRiveraLetelier2015}]\label{lem24}
Let $p$ be a prime and $(k,|\cdot|)$ an ultrametric field of characteristic $p$. Moreover, let $g(\zeta)$ be a parabolic power series in $\Oc_k[[\zeta]]$. Then the following properties hold.
\begin{enumerate}
\item Let $w_0 \neq 0$ in $\mathfrak{m}_k$ be a fixed point of $g$. Then we have \begin{equation}\label{w0} |w_0|\geq |\mu_0(g)|,\end{equation} with equality if and only if \begin{equation}\label{widegw0} \wideg(g(\zeta)-\zeta)= i_0(g)+2.\end{equation} %Moreover, if (\ref{widegw0) holds, then the cycle containing $w_0$ is the only cycle of minimal period
\item Let $n\geq1$ be an integer and $\zeta_0$ in $\mathfrak{m}_k$ a periodic point of $g$ of minimal period $p^n$. If in addition $i_n(g)<+\infty$, then we have \begin{equation}\label{zeta0} |\zeta_0|\geq \left|\frac{\mu_n(g)}{\mu_{n-1}(g)}\right|^{\frac{1}{p^n}},\end{equation} with equality if and only if \begin{equation}\label{widegzeta}\wideg\left(\frac{g^{p^n}(\zeta)-\zeta}{g^{p^{n-1}}(\zeta)-\zeta}\right) = i_n(g)-i_{n-1}(g)+p^n.\end{equation} Moreover, if (\ref{widegzeta}) holds, then the cycle containing $\zeta_0$ is the only cycle of minimal period $p^n$ of $g$ in $\mathfrak{m}_k$, and for every point $\zeta_0'$ in this cycle $|\zeta_0'|=\left|\frac{\mu_n(g)}{\mu_{n-1}(g)}\right|^{\frac{1}{p^n}}$.
\end{enumerate}
\end{lemma}

Assuming Theorem \ref{pncoeff} we now have the results needed to prove Theorem \ref{perpoints}.

%\begin{lemma}[Which one and how to state]\label{bramif}
%If $h(\zeta) = \zeta(1 + \zeta^{b})$ then $i_n(h) = b(1+p+\dots+p^n)$. Perhaps we need Laubie Saine here depending on which route we take.
%\end{lemma}

\begin{proof}[Proof of Theorem \ref{perpoints}] %For fixed points the theorem follows directly from statement 1 of Lemma \ref{lem24}. 
If $\lambda = 0$ the theorem holds trivially. However, if $\lambda \neq0$ then $f$ is 2-ramified by Theorem \ref{pncoeff}, and thus for integers $n\geq0$ we have $\mu_n(f) =  a_2^{p^n-2\frac{p^n-1}{p-1}}\varphi^{\frac{p^n-1}{p-1}}$. Hence, by Lemma \ref{lem24} we have for all 
%integers 
$n\geq1$ and all periodic points $\zeta_0$ in $\mathfrak{m}_k$ of minimal period $p^n$
\begin{align}\label{zeta0ineq}
|\zeta_0| \geq \left|\frac{\mu_n(f)}{\mu_{n-1}(f)}\right|^{\frac{1}{p^n}} &= \left|\frac{a_2^{p^n-2\frac{p^n-1}{p-1}}\varphi^{\frac{p^n-1}{p-1}}}{a_2^{p^{n-1}-2\frac{p^{n-1}-1}{p-1}}\varphi^{\frac{p^{n-1}-1}{p-1}}}\right|^{\frac{1}{p^n}}\notag \\
&= \left|a_2^{(p-3)p^{n-1}}\varphi^{p^{n-1}}\right|^{\frac{1}{p^n}}\notag\\
&= \left|a_2^{p-3}\varphi\right|^{\frac{1}{p}} = |\lambda|^{\frac{1}{p}}.\notag
\end{align}
It follows immediately from Lemma \ref{lem24} %\cite[Lemma 2.4]{LindahlRiveraLetelier2015} 
that for fixed points $\zeta_0\neq 0$ of $f$ in the open unit disk of $k$ we have $|\zeta_0|\geq |a_2|$.
This completes the proof of Theorem \ref{perpoints}.
\end{proof}

\begin{proof}[Proof of Corollary \ref{corrperpoints}]
We note that $\lambda \neq 0$ implies that $f$ is 2-ramified and this in turn implies that for integers $n\geq1$ we have $i_n(f) - i_{n-1}(f) + p^n = 3p^n$. Also, if $\widetilde{f}(\zeta)$ is 3-ramified then for $n\geq1$ we have
\begin{align*}\wideg\left(\frac{f^{p^n}(\zeta)-\zeta}{f^{p^{n-1}}(\zeta)-\zeta}\right) &= \wideg(f^{p^n}(\zeta)-\zeta) - \wideg(f^{p^{n-1}}(\zeta)-\zeta)\\
&= 3\frac{p^n-1}{p-1} - 3\frac{p^{n-1}-1}{p-1} = 3p^n.\end{align*} Hence, (\ref{widegzeta}) in Lemma \ref{lem24} holds with equality. This completes the proof of Corollary \ref{corrperpoints}.
%The proof for power series follows the proof of Theorem \ref{perpoints} for polynomials for all parts except for the use of the Main Lemma in \cite{LindahlRiveraLetelier2015}. By imposing the condition $|a_8|<1$ and $|a_9|<1$ we note that $\widetilde{f}(\zeta) \equiv \zeta(1+a_3\zeta^3) \mod \langle \zeta^{10} \rangle.$ Thus, the Main Lemma of \cite{LindahlRiveraLetelier2015} is still valid as described in Remark \ref{bramif}, and we thus prove our claim.
\end{proof}

%We here state the essential information needed from the Main Lemma of \cite{LindahlRiveraLetelier2015}. XXX PERHAPS WE NEED MORE
%\begin{lemma}[Main Lemma in \cite{LindahlRiveraLetelier2015}]\label{mainlemma}
%Let $p$ be an odd prime, $k$ a field of characteristic $p$, and $b\geq1$ an integer, that is not divisible by $p$, and $0\neq a_b \in k$. Let $g(\zeta)$ be a power series in $k[[\zeta]]$ of the form \[g(\zeta) \equiv \zeta(1+ba_b\zeta^b) \mod \langle \zeta^{3b+1} \rangle.\] Then \[i_n(g) = b(1+p+\dots+p^n).\]
%\end{lemma}
%\begin{rem}\label{rem3ramif}
%The essential information needed from Lemma \ref{mainlemma} is that for $g(\zeta) = \zeta + a\zeta^{4}$, such that $a\neq0$ we know that $i_n(g) = 3(1+p+\dots+p^n)$.
%\end{rem}

\section{Technical results}\label{techres}
 In this section we present results that we use to prove our Main Lemma and Theorem \ref{pncoeff}. %Before proceeding we make the following definition. 

Throughout the paper for any nonnegative integer $n$ let $n!!$ denote the \emph{double factorial} of $n$. We put $0!! :=1$ and $1!! := 1$. For future reference, we note that for integers $n\geq 2$, we have \[n!! = n(n-2)!!.\]

%Let $\Q_p$ be the $p$-adic numbers and let  $\Z_p$ denote its ring of integers. Also let $\nu_p(\cdot)$ denote the $p$-adic valuation.
 
 The proof of the Main Lemma relies on solving linear difference equations expressed as sums of rational functions. For convenience these sums are considered over the $p$-adic numbers $\Q_p$ and its ring of integers $\Z_p$. The main part of this section involves finding the corresponding reductions in $\F_p$. Throughout we let $\nu_p(\cdot)$ denote the $p$-adic valuation.

\begin{mydef} Let $p$ be a prime, and let $f: \Q_p \rightarrow \Q_p$. We say that $f$ has a \emph{pole} at $a$ if $f(a) \notin \Z_p$. Furthermore, if $p^{n}f(a) \in \Z_p$ but $p^{n-1}f(a) \notin \Z_p$, then we say that the pole is of order $n$. A pole of order 1 is called a simple pole. Moreover, we define the \emph{residue} of a function at a pole $a$ of order $n$ as $p^nf(a)$.
\end{mydef}
%The following remark 
\begin{lemma}\label{pole}
Let $p$ be a prime and let $f:\Q_p\rightarrow\Q_p$ be such that $f$ only has simple poles. Put \[ F := p\sum_{j=n}^N f(j).\] Then the reduction $\widetilde{F}$ is well-defined and  \[\widetilde{F} = \sum_{\substack{n\leq a\leq N\\\text{$a$ is a pole of $f$}}}pf(a).\]
\end{lemma}
\begin{proof}
Using that $f$ only has simple poles we see that the reduction of $pf(a)$ is well-defined, and the proof of the lemma follows from seeing that for all elements $b \in [n,N]$ such that $b$ is not a pole of $f$ we have $pf(b) \equiv 0 \mod {p\Z_p}$.
\end{proof}
\begin{lemma}\label{generalwilson}
Let $p$ be a prime, $a,b \in \mathbb{F}_p$, and $a \neq0$. Furthermore let $f(n) = an + b$, $s' \equiv -a^{-1}b\pmod{p}$ and $S = \mathbb{F}_p \setminus \{s'\}$. Then \[\prod_{s \in S} f(s) \equiv -1 \pmod{p}.\]
\end{lemma}

\begin{proof}
This is a consequence of Wilson's theorem, and the fact that any linear function defined on $\F_p$ simply permutes the elements of $\F_p$.
%Any linear function defined over $\mathbb{F}_p$ simply permutes the elements, since $an_1 + b \equiv an_2+ b \pmod{p}$ if and only if $n_1 \equiv n_2 \pmod{p}$. That means when excluding the unique mapping $ak+b$ that maps to zero the remaining $p-1$ factors is simply a permutation of the elements in $\mathbb{F}_p^\times$ and by Wilson's theorem we know that their product is congruent to $-1 \pmod{p}$.
\end{proof}

The following two lemmas are slightly reformulated versions of Lemma 2 and 3 in \cite{Fransson2017}. 

\begin{lemma}\label{sumident} Let $p$ be an odd prime. For each integer $n\geq1$ let $\mathcal{R}_n$ and $\mathcal{T}_n$ in $\Q_p$ be defined by
\begin{equation}\label{Rdef}\mathcal{R}_n:=(2n-1)!!\sum_{r = 1}^{n}\left[\prod_{j=r+1}^{n}\frac{2j}{2j-1}\right],\end{equation}
and
\begin{equation}\label{Tdef}\mathcal{T}_n:=(2n+1)!!\sum_{j=1}^n\frac{(2j)!!}{(2j+1)!!}.
\end{equation}
%and we let $p$ be a prime then let $T$ be the series on $\Q_p$ defined as
%\item \begin{equation*} T = (2p+1)!!\sum_{j=1}^p\frac{(2j)!!}{(2j+1)!!}\end{equation*}
Then \[\mathcal{R}_n = (2n+1)!! - (2n)!! \text{ and }\mathcal{T}_n = (2n+2)!! - 2(2n+1)!!.\] In particular, $\widetilde{\mathcal{R}}_p = \widetilde{\mathcal{T}}_p = 0, \widetilde{\mathcal{R}}_{p-1}= 1, \widetilde{\mathcal{T}}_{p-1} = 0$ and $\widetilde{\mathcal{R}}_{p-2} = -1/2, \widetilde{\mathcal{T}}_{p-2} = -1$.
\end{lemma}
\begin{proof}
All parts of the proof except for the evaluation of $\widetilde{\mathcal{R}}_n$ and $\widetilde{\mathcal{T}}_n$ at $n=p-1$ and $n=p-2$ are given in \cite{Fransson2017}. Thus to complete the proof we compute $\widetilde{\mathcal{R}}_{p-1}, \widetilde{\mathcal{R}}_{p-2}, \widetilde{\mathcal{T}}_{p-1}$ and $\widetilde{\mathcal{T}}_{p-2}$. We note that $(2n)!! = 2^nn!$, and that $\nu_p((2p-3)!!) = 1$ and thus we have $\widetilde{\mathcal{R}}_{p-1} = 0 - 2^{p-1}(p-1)! = 0-(-1),$ and $\widetilde{\mathcal{R}}_{p-2} = 0-2^{p-2}(p-2)! = -1/2.$ 
Also note that $\nu_p((2p)!!) = \nu_p((2p-1)!!)=1$. Consequently  $\widetilde{\mathcal{T}}_{p-1} = 0,$  and $\widetilde{\mathcal{T}}_{p-2} = (2p-2)!!-0 = -1$.
\end{proof}

\begin{lemma}\label{sumidentfinitefield}
Let $p$ be an odd prime and let $a,b$ be integers. For every integer $n\geq1$ let $\mathcal{S}_n(a,b)$ in $\Q_p$ be defined by
\begin{equation}\label{Sdef}\mathcal{S}_n(a,b) := (2n+1)!!\sum_{j=1}^n\frac{a j + b}{2j+1}.\end{equation}
Then $\mathcal{S}_n(a,b) \in \mathbb{Z}_p$ and in particular $\widetilde{\mathcal{S}}_p(a,b)  = \widetilde{\mathcal{S}}_{p-1}(a,b)= a/2 - b$, and $\widetilde{\mathcal{S}}_{p-2} = -a/2+b.$
%for each $n\geq (p-1)/2$ we define $\mathcal{C} = (2n+1)!!/p$. Then $\mathcal{\widetilde{S}}_n(a,b) = \widetilde{\mathcal{C}}(-a/2+b)$ 
\end{lemma}

\begin{proof}
As for the proof of Lemma \ref{sumident} the proof is given in \cite{Fransson2017} except for the calculations of $\widetilde{\mathcal{S}}_{p-1}$ and $\widetilde{\mathcal{S}}_{p-2}$. Note that for $j$ in $[1,p-1]$ the function $f(j)=(aj+b)/(2j+1)$ has exactly one pole. This pole occurs for  $j=(p-1)/2$ and is of order one. Moreover, by Lemma \ref{generalwilson},  $\frac{(2p-1)!!}{p} = -1 \pmod{p}.$ Consequently by Lemma \ref{pole} $\widetilde{\mathcal{S}}_{p-1} = -(\frac{a(p-1)}{2}+b),$ and $\widetilde{\mathcal{S}}_{p-2} = (\frac{a(p-1)}{2}+b).$
\end{proof}

The main idea behind finding the reductions of the sums of rational functions we set out to find, is to compute and sum over its residues. In Lemma \ref{lemmaevenfact}, \ref{polydoublefact} and \ref{doublefactinner} we discuss three different types of functions that will repeatedly occur in later lemmas, and we address how to compute their corresponding reductions, or in the case of Lemma \ref{doublefactinner}, how to express it in simpler terms.

\begin{lemma}\label{lemmaevenfact} Let $p$ be an odd prime. Let $f:\mathbb{Z}_p \rightarrow \mathbb{Z}_p$, and for each integer $n\geq1$ let $\mathcal{F}_n \in \mathbb{Q}_p$ be defined as
 \[\mathcal{F}_n = (2n+2)!!\sum_{j=1}^{n-1} \frac{f(j)}{(2j+4)!!}.\] Then $\mathcal{F}_n\in \Z_p$ and \[\widetilde{\mathcal{F}}_p = 2\widetilde{f}(p-2)+\widetilde{f}(p-1).\]
\end{lemma}

\begin{lemma}\label{polydoublefact} Let $p\geq7$ be a prime. Let $q \in \mathbb{Z}_p[x]$ be a polynomial such that $\deg(q) < (p+3)/2$, and for each integer $n\geq1$ let $\mathcal{F}_n' \in \mathbb{Q}_p$ be defined as
 \[\mathcal{F}_n' := (2n+3)!!\sum_{j=1}^{n-1} \frac{q(j)(2j)!!}{(2j+5)!!}.\] Then $\mathcal{F}_n' \in \mathbb{Z}_p$ and  $\widetilde{\mathcal{F}}_p' = 0.$
\end{lemma}

\begin{lemma}\label{doublefactinner}
Let $p$ be an odd prime and let $f:\mathbb{Q}_p\rightarrow \mathbb{Q}_p$. Furthermore, for any integer $n\geq 1$ let $\mathcal{F}_n''$ in $\mathbb{Q}_p$ be defined as
\[\mathcal{F}_n'' :=(2n+3)!!\sum_{j=1}^{n-1}\frac{(2j+4)!!}{(2j+5)!!}\sum_{i=1}^{j-1}f(i).\]
Then \[\mathcal{F}_n'' = (2n+4)!!\sum_{i=1}^{n-1}f(i) - (2n+3)!!\sum_{i=1}^{n-1}f(i)\frac{(2i+6)!!}{(2i+5)!!}.\]
\end{lemma}

\begin{proof}[Proof of Lemma \ref{lemmaevenfact}]
The lemma is a direct consequence of the fact that by Lemma \ref{pole}
%For $n=p$ the rational function in the sum has two simple poles: $j=p-1$ and $j=p-2$, and no others. Hence,  
\[\mathcal{F}_p \equiv (2p+2)!!\left(\frac{f(p-2)}{(2p)!!} + \frac{f(p-1)}{(2p+2)!!}\right) \mod p\mathbb{Z}_p.\]\end{proof} 
%and we  conclude that \[\mathcal{\widetilde{F}}_p = 2\widetilde{f}(p-2)+\widetilde{f}(p-1).\]

The proof of Lemma \ref{polydoublefact} and Lemma \ref{doublefactinner} follows after the statement and proof of the following lemma, which will be utilized in the proof of Lemma \ref{polydoublefact}.

\begin{lemma}\label{priordoublefact} Let $p \geq 7$ be a prime. Furthermore, let $n\geq2$ and $k\geq0$ be integers, and define $\mathcal{K}(n,k)$ in $\Q_p$
 as \[\mathcal{K}(n,k) :=\sum_{j=1}^{n-1} \frac{(2(j+k))!!}{(2j+5)!!}.\] Then \begin{equation}\label{mathfrakS}\mathcal{K}(n,k) = \frac{(2(n+k))!!}{(2k-3)(2n+3)!!}-\frac{(2k+2)!!}{(2k-3)5!!}.\end{equation} In particular for $k < (p+3)/2$, we have \begin{equation}\label{dubbelfactkmodp}\mathcal{K}(p,k) \in\Z_p.\end{equation}
\end{lemma}

\begin{proof}
We proceed by induction in $n$. Note that for $n=2$, by definition%(\ref{mathfrakS}) is valid for $n=2$, since 
\[\mathcal{K}(2,k) = \sum_{j=1}^{1}\frac{(2(j+k))!!}{(2j+5)!!} = \frac{(2k+2)!!}{7!!}.\] On the other hand \[\frac{(2(2+k))!!}{(2k-3)7!!}-\frac{(2k+2)!!}{(2k-3)5!!} = \frac{(2k+2)!!}{(2k-3)5!!}\left(\frac{(2k+4)}{7}-1\right) = \frac{(2k+2)!!}{7!!}.\] This proves (\ref{mathfrakS}) for $n=2$. Assume that (\ref{mathfrakS}) is valid for $n\geq2$. Then 
\begin{align*}
\mathcal{K}(n+1,k) &= \sum_{j=1}^{n} \frac{(2(j+k))!!}{(2j+5)!!}\\
&= \frac{(2(n+k))!!}{(2n+5)!!} + \sum_{j=1}^{n-1} \frac{(2(j+k))!!}{(2j+5)!!}.\\
\end{align*}
By the induction hypothesis we obtain
\begin{align*}
\mathcal{K}(n+1,k)&=  \frac{(2(n+k))!!}{(2n+5)!!} + \frac{(2(n+k))!!}{(2k-3)(2n+3)!!}-\frac{(2k+2)!!}{(2k-3)5!!}\\
&= \frac{(2(n+k))!!(2k-3) + (2(n+k))!!(2n+5)}{(2k-3)(2n+5)!!}-\frac{(2k+2)!!}{(2k-3)5!!}\\
&= \frac{(2(n+k))!!(2k-3 +2n+5)}{(2k-3)(2n+5)!!}-\frac{(2k+2)!!}{(2k-3)5!!}\\
%&= \frac{(2(n+k))!!(2(n+k)+2)}{(2k-3)(2n+5)!!}-\frac{(2k+2)!!}{(2k-3)5!!}\\
&= \frac{(2(n+k)+2)!!}{(2k-3)(2n+5)!!}-\frac{(2k+2)!!}{(2k-3)5!!},\\
\end{align*}
which completes the induction step. The second statement of the lemma follows by letting $n=p$ in (\ref{mathfrakS}). Clearly $\nu_p((2p+2k)!!) = \nu_p((2p+5)!!)=1$. Moreover for $k<(p+3)/2$ we have $\nu_p(2k-3) = 0$. Accordingly, $\mathcal{K}(p,k)$ is in $\Z_p$ for all $k < (p+3)/2$ by definition. %This completes the proof of Lemma \ref{priordoublefact}.
%We obtain
%\[p\mathcal{K}(p,k) = p\left(\frac{(2p +2k)!!}{(2k-3)(2p+3)!!} - \frac{(2k+2)!!}{(2k-3)5!!}\right).\] Since we assume that $p>5$ we see that for $k=(p+3)/2$ we obtain \[\mathcal{K}(p,(p+3)/2)= \left(\frac{(3p+3)!!}{(p)(2p+3)!!} - \frac{(p+5)!!}{(p)5!!}\right),\] and thus $\mathcal{K}(p,k)$ has a simple pole. However, for all $k<(p+3)/2$ its reduction is zero.
\end{proof}

%The proof follows directly after the following remark.
%\begin{rem}
%By direct computation we note that for $p=3$ we have that $\mathcal{F}_3 = q(2)$, regardless of the degree of $q$. Hence, for all a purposes of using this lemma, it is sufficient to note that $\widetilde{q}(2)=0$ for $p=3$.
%\end{rem}
\begin{proof}[Proof of Lemma \ref{polydoublefact}] For all $1\leq j\leq n-1$ we note that $\nu_p((2n+3)!!) \geq \nu_p((2j+5)!!)$ and thus $\mathcal{F}_n' \in \Z_p$. Put \[q_k(j) := \begin{cases}1, \text{ if $k=0$}\\{\displaystyle \prod_{i=1}^k (2(j+i)),\text{ if $k\geq1$}}.\end{cases}\]
Then by definition \[\mathcal{K}(p,k) = \sum_{j=1}^{p-1}\frac{(2j)!!}{(2j+5)!!}q_k(j).\]
%First we note that $\mathcal{F}_n' \in \mathbb{Z}_p$ since the denominator is cancelled in every term by the outer factor. Furthermore we want to show that $\widetilde{\mathcal{F}}_p' = 0$, and thus can without loss of generality change the factor $(2p+3)!!$ for $p$, since $(2p+3)!!$  only contains one factor $p$ whenever $p>3$. By Lemma \ref{priordoublefact} we know that $p\mathcal{K}(p, k) \equiv 0 \mod p\Z_p$ if $k < (p+3)/2$. We also note that for some $k < (p+3)/2$ we have \begin{align*}\mathcal{K}(p,k) &= \sum_{j=1}^{p-1} \frac{(2(j+k))!!}{(2j+5)!!}\\
% &= \sum_{j=1}^{p-1} \frac{(2j)!!}{(2j+5)!!}(2j+2k)(2j+2k-2)\cdots(2j+2k-(2k-2))\\
 % &= \sum_{j=1}^{p-1} \frac{(2j)!!}{(2j+5)!!}w_k(j).
 %\end{align*}
 From the definition of $q_k(j)$ we have $\deg(q_k(j)) =k$, and we write $q_k(j) = b_k^{(k)} j^k + b_{k-1}^{(k)}j^{k-1} + \dots +b_1^{(k)}j + b_0^{(k)}$. Moreover, we define \[
\mathcal{G}_k := \sum_{j=1}^{p-1}j^k \frac{(2j)!!}{(2j+5)!!}.\] Hence,
 \begin{align}\label{lineardep}\mathcal{K}(p,k)   %&= \sum_{j=1}^{p-1} \frac{(2j)!!}{(2j+5)!!}w_k(j)\notag\\
&= \sum_{j=1}^{p-1} \frac{(2j)!!}{(2j+5)!!}(b_k^{(k)} j^k +  \dots + b_0^{(k)})\notag\\
&=b_k^{(k)}\mathcal{G}_k + \dots + b_0^{(k)}\mathcal{G}_0.
 \end{align}
Assume that $k<(p+3)/2$. Then by Lemma \ref{priordoublefact} we have% from  (\ref{lineardep}) multiplied by $p$ we obtain 
 %\begin{equation}\label{lineardepmodp}0=p\widetilde{\mathfrak{S}}(p,k) = p\left(\widetilde{b}_k^{(k)}\widetilde{\mathcal{G}}_k + \widetilde{b}_{k-1}^{(k)}\widetilde{\mathcal{G}}_{k-1} + \dots +\widetilde{b}_1^{(k)}\widetilde{\mathcal{G}}_1+\widetilde{b}_0^{(k)}\widetilde{\mathcal{G}}_0\right).\end{equation}
 \begin{equation}\label{lineardepmodp}0\equiv p\mathcal{K}(p,k) \equiv p\left(b_k^{(k)}\mathcal{G}_k + \dots +b_0^{(k)}\mathcal{G}_0\right) \mod p\Z_p.\end{equation}
 We know by Lemma \ref{priordoublefact} that $\widetilde{p\mathcal{G}}_0 = 0$. Consequently, for $k=1$ (\ref{lineardepmodp}) implies that $\widetilde{p\mathcal{G}}_1 = 0$ %indeIn view of (\ref{lineardepmodp}) with $k = 1$ since $\widetilde{p\mathcal{G}}_0 =0$ then $\widetilde{p\mathcal{G}}_1 =0$ 
 independently of $b_1^{(1)}$. Inductively $\widetilde{p\mathcal{G}}_k =0$ for all $k<(p+3)/2$ independently of the coefficient $b_k^{(k)}$. Accordingly, (\ref{lineardepmodp}) holds even if the polynomial $q_k(j)$ is interchanged by any polynomial $q$ of degree strictly less than $(p+3)/2$. This completes the proof of Lemma \ref{polydoublefact}.
\end{proof}

\begin{proof}[Proof of Lemma \ref{doublefactinner}]
By interchanging the order in the summations in $\mathcal{F}_n''$ we obtain
\begin{equation}\label{startsum}\mathcal{F}_n'' = (2n+3)!!\sum_{i=1}^{n-1}f(i)\sum_{j=i+1}^{n-1}\frac{(2j+4)!!}{(2j+5)!!}.\end{equation} From Lemma \ref{sumident} we have \begin{align}\label{rewritetn}
\sum_{j=i+1}^{n-1} \frac{(2j+4)!!}{(2j+5)!!} &= \sum_{j=1}^{n-1} \frac{(2j+4)!!}{(2j+5)!!} - \sum_{j=1}^{i} \frac{(2j+4)!!}{(2j+5)!!} \notag\\
&= \frac{(2n+4)!!-2(2n+3)!!}{(2n+3)!!}-\frac{(2i+6)!!-2(2i+5)!!}{(2i+5)!!}.
\end{align} 
Insertion of (\ref{rewritetn}) in (\ref{startsum}) yields 
\begin{align*}
\mathcal{F}_n'' &=(2n+3)!!\sum_{i=1}^{n-1}f(i)\left(\frac{(2n+4)!!-2(2n+3)!!}{(2n+3)!!}-\frac{(2i+6)!!-2(2i+5)!!}{(2i+5)!!}\right)\\
&=((2n+4)!!-2(2n+3)!!)\sum_{i=1}^{n-1}f(i) \\
&\hspace{5mm}- (2n+3)!!\sum_{i=1}^{n-1}f(i)\frac{(2i+6)!!-2(2i+5)!!}{(2i+5)!!}\\
&=(2n+4)!!\sum_{i=1}^{n-1}f(i) - (2n+3)!!\sum_{i=1}^{n-1}f(i)\frac{(2i+6)!!}{(2i+5)!!}.
\end{align*}
%This completes the proof of Lemma \ref{doublefactinner}.
\end{proof}

%\begin{lemma}\label{lemmaoddfact2} Let $p$ be an odd prime. Let $f:\mathbb{Z}_p \rightarrow \mathbb{Z}_p$, and for each integer $n\geq1$ let $\mathcal{F}^*_n \in \mathbb{Z}_p$ be defined as
% \[\mathcal{F}^*_n = (2n+3)!!\sum_{j=1}^{n-1} \frac{f(j)}{(2j+5)!!}\] then the reduction of $\mathcal{F}^*_p$ is determined by \[-3\sum_{s\in S} \frac{f(s)}{(2s+5)!!/p},\] where $S$ is the set \[\{(p-5)/2,\dots,p-1\}\setminus \{n :  f(n) \equiv 0 \mod p\Z_p\}.\]
%\end{lemma}

The upcoming Lemma \ref{dmcoeff}, Lemma \ref{Ssum} and Lemma \ref{uvxw} address the reductions of the specific functions that are considered in the proof of the Main Lemma.

\begin{lemma}\label{dmcoeff}Let $p$ be an odd prime. For each integer $n\geq1$, let $\mathcal{U}_n, \mathcal{V}_n, \mathcal{W}_n$ and $\mathcal{X}_n$ in $\mathbb{Q}_p$ be defined by
\begin{equation}\label{Udef}\mathcal{U}_n(a,b) := (2n+2)!!\sum_{j=1}^{n-1}\frac{\mathcal{S}_j(a,b)(2j+3)}{(2j+4)!!},\end{equation}
\begin{equation}\label{Vdef}\mathcal{V}_n(a,b,c) := (2n+2)!!\sum_{j=1}^{n-1}\frac{\mathcal{R}_j\cdot(aj^2+bj+c)}{(2j+4)!!},\end{equation}
\begin{equation}\label{Wdef}\mathcal{W}_n := (2n+2)!!\sum_{j=1}^{n-1}\frac{\mathcal{T}_j\cdot(2j+3)}{(2j+4)!!},\end{equation}
and
\begin{equation}\label{Xdef}\mathcal{X}_n(a,b) := (2n+2)!!\sum_{j=1}^{n-1}\frac{(aj+b)(2j+1)!!}{(2j+4)!!}.\end{equation}
Then \[\widetilde{\mathcal{U}}_p(a,b) = 3/2a -3b, \quad\widetilde{\mathcal{V}}_p(a,b,c) = -3a+b, \quad \widetilde{\mathcal{W}}_p = 2 \text{ and } \widetilde{\mathcal{X}}_p = 0.\]
\end{lemma}

\begin{proof}
The proof relies on repeated use of Lemma \ref{lemmaevenfact}.
% and the definition of the inner functions $\mathcal{S}_n, \mathcal{R}_n$ and $\mathcal{T}_n$. 
For $\mathcal{U}_n$ we have \[\widetilde{\mathcal{U}}_p(a,b) = 2(2p-1)\mathcal{\widetilde{S}}_{p-2}(a,b)+(2p+1)\mathcal{\widetilde{S}}_{p-1}(a,b) = \mathcal{\widetilde{S}}_{p-1}(a,b)-2\mathcal{\widetilde{S}}_{p-2}(a,b).\] By Lemma \ref{sumidentfinitefield}, 
$\mathcal{\widetilde{S}}_{p-1}(a,b) = a/2-b$ and $\widetilde{\mathcal{S}}_{p-2} = -a/2+b$, so that $\mathcal{\widetilde{U}}_p(a,b) = 3/2a-3b.$

For $\mathcal{V}_n$ we have \[\widetilde{\mathcal{V}}_p(a,b,c) = 2(2p-1)(4a-2b+c)\mathcal{\widetilde{R}}_{p-2}+(2p+1)(a-b+c)\mathcal{\widetilde{R}}_{p-1}.\] By Lemma \ref{sumident} we obtain the reduction
% in Lemma \ref{sumident} for $n\geq (p-1)/2$ we have $\mathcal{\widetilde{R}}_n = -(2n)!!$. Hence,
\[\widetilde{\mathcal{V}}_p(a,b,c) = (-2)(4a-2b+c)(2p-4)!! + (a-b+c)(2p-2)!!.\] 
From Lemma \ref{generalwilson} the reductions of $(2p-2)!!$ and $(2p-4)!!$ are $-1$ and $1/2$. Hence,
 \[\mathcal{\widetilde{V}}_p(a,b,c) = -3a+b.\]
%Let $c_n = (2n)!!$. From Lemma \ref{generalwilson} we get $\widetilde{c}_{p-1} = -1$ which implies $\widetilde{c}_{p-2} = 1/2$
In a similar way we deduce by Lemma \ref{lemmaevenfact} that for $\mathcal{W}_n$ we have \[\mathcal{\widetilde{W}}_p = 2(2p-1)\mathcal{\widetilde{T}}_{p-2}+(2p+1)\mathcal{\widetilde{T}}_{p-1},\] 
%By the definition of $\mathcal{T}_n$ in Lemma \ref{sumident} for $n\geq (p-1)/2$ we have $\mathcal{\widetilde{T}}_n = (2n+2)!!$. 
and by Lemma \ref{sumident} it follows that \[\mathcal{\widetilde{W}}_p = (-2)(-1) = 2.\] Finally, concerning $\mathcal{X}_p$ we note that  \[\nu_p\left(\frac{(2j+1)!!}{(2j+4)!!}\right)\geq0.\] Hence, $\mathcal{\widetilde{X}}_p(a,b) = 0$. This completes the proof of Lemma \ref{dmcoeff}.
\end{proof}
%The following Lemma will be used in the proof of Lemma \ref{Ssum} and Lemma \ref{uvxw}.
\begin{lemma}\label{harmonic}
Let $p\geq5$ be a prime, and $\gamma = \frac{p-3}{2}$. Furthermore let $n\geq1$ be an integer, $\mathcal{H}_n$ and $\mathcal{H}_n'$ in $\mathbb{Q}_p$ be defined as \[\mathcal{H}_n := \sum_{j=1}^n \frac{1}{2j+1}\] and \[\mathcal{H}_n':= \sum_{j=1}^n \frac{1}{2j},\] then $\mathcal{H}_\gamma, \mathcal{H}_\gamma' \in \mathbb{Z}_p$ and $\mathcal{\widetilde{H}}_\gamma + \mathcal{\widetilde{H}}_\gamma' = 0.$
\end{lemma}

\begin{proof}
%By Wolstensholme's theorem on harmonic sums modulo $p$ we have \[\sum_{j=1}^{p-1} \frac{1}{j} \equiv 0 \pmod{p}.\] Hence \[\mathcal{H}_\gamma + \mathcal{H}_\gamma' = \sum_{j=2}^{p-2}\frac{1}{j} = \sum_{j=1}^{p-1}\frac{1}{j} - \left(1+ \frac{1}{p-1}\right) \equiv 0 \pmod{p},\] which completes the proof of the lemma.
By definition $\mathcal{H}_\gamma$ and $\mathcal{H}_\gamma'$ are in $\mathbb{Z}_p$, and also note that the sum of the multiplicative inverse elements of $1$ and $p-1$ is 0. Thus \[\mathcal{H}_\gamma + \mathcal{H}_\gamma' = \sum_{j=1}^{p-1} j^{-1}.\] $\mathcal{H}_\gamma + \mathcal{H}_\gamma'$ is equal to the sum of all elements in $\mathbb{F}_p^\times$. Consequently, \[\mathcal{H}_\gamma+ \mathcal{H}_\gamma' \equiv 0 \pmod{p}.\] %This completes the proof of the lemma.
\end{proof}

\begin{lemma}\label{Ssum}
Let $p\geq 7$ be a prime. For each $n\geq1$ define $\widehat{\mathcal{S}}_n, \widehat{\mathcal{T}}_n, \widehat{\mathcal{R}}_n$, and $\mathcal{Z}_n$ in $\Z_p$ as 
\begin{equation}\label{sumsumsum}\widehat{\mathcal{S}}_n(a,b,c) := (2n+3)!!\sum_{j=1}^{n-1}\frac{(cj+1)(2j+3)}{(2j+5)!!}\mathcal{S}_j(a,b),\end{equation}
\begin{equation}\label{sumsumsum2}\widehat{\mathcal{T}}_n(a) := (2n+3)!!\sum_{j=1}^{n-1}\frac{(aj+1)(2j+3)}{(2j+5)!!}\mathcal{T}_j,\end{equation}
\begin{equation}\label{sumsumsum3}\widehat{\mathcal{R}}_n(a) := (2n+3)!!\sum_{j=1}^{n-1}\frac{(2j+2)(aj+1)}{(2j+5)!!}\mathcal{R}_j,\end{equation}
\begin{equation}\label{sumsumsum4}\mathcal{Z}_n(a,b,c) := (2n+3)!!\sum_{j=1}^{n-1}\frac{aj^2+bj+c}{(2j+3)(2j+5)}.\end{equation} 
%\begin{equation}\label{sumsumsum4}\psi_n := (2n+3)!!\sum_{j=1}^{n-1}\frac{j(2j+1)!!}{(2j+5)!!}.\end{equation} 
Then \[\widetilde{\widehat{\mathcal{S}}}_p(a,b,c) = 1/4(a(17-41c)+b(-4+7c)), \quad \widetilde{\widehat{\mathcal{T}}}_p(a) = 6-15a\] \[\widetilde{\widehat{\mathcal{R}}}(a)_p = 9a-3 , \text{ and } \widetilde{\mathcal{Z}}(a,b,c)_p = 6a-\frac{3}{2}b.\]
\end{lemma}

\begin{proof}
\textbf{Proof of $\widehat{\mathcal{S}}_n$}:
By the definition of $\mathcal{S}_j$ we obtain \[\widehat{\mathcal{S}}_n(a,b,c) = (2n+3)!!\sum_{j=1}^{n-1}\frac{cj+1}{2j+5}\sum_{i=1}^j \frac{ai+b}{2i+1}.\] Let $p\geq7$ and put $n=p$. Put 
\begin{equation}\label{sum1} \mathcal{Y}_1 := (2p+3)!!\left(\frac{c(p-5)+2}{2p}\sum_{i=1}^{(p-5)/2} \frac{ai+b}{2i+1}\right),\end{equation}
and
\begin{equation}\label{sum2}\mathcal{Y}_2:=(2p+3)!!\left(\frac{(a(p-1)+2b}{2p}\sum_{j=(p-1)/2}^{p-1}\frac{cj+1}{2j+5}\right).\end{equation}
Note that the only poles of $\widehat{\mathcal{S}}_p(a,b,c)$ occur for $j=(p-5)/2$ and $i=(p-1)/2$. Accordingly by Lemma \ref{pole} 
% then we note that for each $j < (p-5)/2$ each term will vanish in the reduction, since $\nu_p((2j+5)!!=0$, i.e. 
%\[\widehat{\mathcal{S}}_p(a,b,c) \equiv (2p+3)!!\sum_{j=(p-5)/2}^{p-1}\frac{cj+1}{2j+5}\sum_{i=1}^j \frac{ai+b}{2i+1} \mod p\mathbb{Z}_p.\]
%However, for $j = (p-5)/2$ one term will remain after the reduction namely 
%\begin{equation}\label{sum1} \mathcal{Y}_1 := (2p+3)!!\left(\frac{c(p-5)+2}{2p}\sum_{i=1}^{(p-5)/2} \frac{ai+b}{2i+1}\right).\end{equation}
%Then for each $j\geq (p-1)/2$ there will be a $p$ in the denominator in the inner sum. However, that specific term corresponding to $i = (p-1)/2$ is the only nonzero term in the reduction in the inner sum. Therefore for each $j\geq (p-1)/2$ we can associate the term of $i = (p-1)/2$ to the value of the outer sum which then corresponds to \begin{equation}\label{sum2}\mathcal{Y}_2:=(2p+3)!!\left(\frac{(a(p-1)+2b}{2p}\sum_{(p-1)/2}^{p-1}\frac{cj+1}{2j+5}\right),\end{equation}
% these are the only nonzero cases of the reduction. Hence, 
\begin{equation}\label{gammap} \widehat{\mathcal{S}}_p(a,b,c) \equiv \mathcal{Y}_1 + \mathcal{Y}_2 \mod p\Z_p.\end{equation} To finish the proof is to show that $\mathcal{\widetilde{Y}}_1 +  \mathcal{\widetilde{Y}}_2 = 1/4(a(17-41c)+b(-4+7c))$.

For the sum in (\ref{sum1}) we note that \[\sum_{i=1}^{(p-5)/2} \frac{ai+b}{2i+1} = \sum_{i=1}^{(p-3)/2} \frac{ai+b}{2i+1} - \frac{a(p-3)+2b}{2p-4}.\] Moreover,  
\begin{align}\label{sum1done2}\sum_{i=1}^{(p-3)/2}\frac{ai+b}{2i+1}  &=\sum_{i=1}^{(p-3)/2}\frac{a}{2} -  \left(\frac{a}{2}-b\right)\frac{1}{2i+1}\notag\\
&= \frac{a(p-3)}{4} - \left(\frac{a}{2}-b\right)\sum_{i=1}^{(p-3)/2}\frac{1}{2i+1}.\notag
 \end{align}
 By the definition of $\mathcal{H}_\gamma$ in Lemma \ref{harmonic} we then have \begin{equation}\label{sum1done}
\sum_{i=1}^{(p-3)/2}\frac{ai+b}{2i+1} = \frac{a(p-3)}{4} - \left(\frac{a}{2}-b\right)\mathcal{H}_\gamma.\end{equation}
It follows that \begin{equation}\label{y1done}\mathcal{Y}_1 = \frac{(2p+3)!!}{p}\cdot\frac{c(p-5)+2}{2}\left(\frac{a(p-3)}{4}-\left(\frac{a}{2}-b\right)\mathcal{H}_\gamma - \frac{a(p-3)+2b}{2p-4}\right).\end{equation}

By a switch of index in the sum of (\ref{sum2}) we obtain
\begin{align}\label{sumsum2}
\sum_{j=(p-1)/2}^{p-1}\frac{cj+1}{2j+5} &= \sum_{k=1}^{(p+1)/2}\frac{c(2k+p-3)+2}{2(2k+p+2)} \notag\\
&= \left(\sum_{k=1}^{(p+3)/2}\frac{c(2k+p-5)+2}{2(2k+p)}\right) - \frac{c(p-3)+2}{2p+4}\notag\\
&= \frac{1}{2}\left(\sum_{k=1}^{(p-3)/2}\frac{c(2k + p-5)+2}{2k+p}\right)\notag\\
&\hspace{5mm} - \frac{c(p-3)+2}{2p+4} + \frac{c(2p-6)+2}{4p-2} + \frac{c(2p-4)+2}{4p+2} + \frac{c(2p-2)+2}{4p+6}.\notag
\end{align}
Put \[\mathcal{P} := \frac{1}{2}\sum_{k=1}^{(p-3)/2}\frac{c(2k + p-5)+2}{2k+p} = \frac{1}{2}\sum_{k=1}^{(p-3)/2}\left(c+\frac{2-5c}{2k+p}\right) = \frac{c(p-3)}{4} + \frac{2-5c}{2}\sum_{k=1}^{(p-3)/2}\frac{1}{2k+p}.\] and \[\sigma := - \frac{c(p-3)+2}{2p+4} + \frac{c(2p-6)+2}{4p-2} + \frac{c(2p-4)+2}{4p+2} + \frac{c(2p-2)+2}{4p+6}.\]
%&= \sigma + \frac{1}{2}\sum_{k=1}^{(p-3)/2}\frac{c(2k + p-5)+2}{2k+p},
%where $\sigma$ is the sum of the four terms not included in the sum. 
Note that by definition \begin{equation}\label{y2done}\mathcal{Y}_2 = +\frac{(2p+3)!!}{p}\cdot\frac{a(p-1)+2b}{2}\left(\mathcal{P} + \sigma \right).\end{equation}
In view of Lemma \ref{harmonic} we have \begin{equation}\label{ptilde}\mathcal{\widetilde{P}} =-\frac{1}{2}(3/2c+(5c-2)\mathcal{\widetilde{H}}_\gamma').\end{equation}
We  also note that $\widetilde{\sigma} = (17c-2)/12$.%Let $\mathcal{P}$ denote the sum \[\mathcal{P} := \frac{1}{2}\sum_{k=1}^{(p-3)/2}\frac{c(2k + p-5)+2}{2k+p} = \frac{1}{2}\sum_{k=1}^{(p-3)/2}\left(c+\frac{2-5c}{2k+p}\right) = \frac{c(p-3)}{4} + \frac{2-5c}{2}\sum_{k=1}^{(p-3)/2}\frac{1}{2k+p}.\]In view of Lemma \ref{harmonic} \begin{equation}\label{ptilde}\mathcal{\widetilde{P}} =-\frac{1}{2}(3/2c+(5c-2)\mathcal{\widetilde{H}}_\gamma').\end{equation} %and (\ref{sumsum2}) simply reads $\mathcal{P}+\mu$. %Since the denominator is nonzero in the reduction for all $k$ in $\mathcal{P}$ the reduction is well-defined, and each term $(2k+p)^{-1}$ can be associated to its representative in $\{0,1,2,\dots, p-1\}$ which means $(2k+p)^{-1} \equiv (2k)^{-1} \pmod{p}.$ 

%Combining (\ref{sum1}) and (\ref{sum2}) together with (\ref{sum1done}) and (\ref{sumsum2}) yields 
%\begin{align*} \mathcal{Y}_1 + \mathcal{Y}_2 &= \frac{(2p+3)!!}{p}\cdot\frac{c(p-5)+2}{2}\left(\frac{a(p-3)}{4}-\left(\frac{a}{2}-b\right)\mathcal{H}_\gamma - \frac{a(p-3)+2b}{2p-4}\right)\\
%&\hspace{5mm}+\frac{(2p+3)!!}{p}\cdot\frac{a(p-1)+2b}{2}\left(\mathcal{P} + \sigma \right).
%\end{align*}
Together with (\ref{y1done}), (\ref{y2done}) and using that the reduction of $(2p+3)!!/p$ is $-3$, we obtain 
\begin{align*}
\mathcal{\widetilde{Y}}_1+\mathcal{\widetilde{Y}}_2&= -\frac{3}{2}(2-5c)\left(-\frac{3a}{4} - \left(\frac{a}{2}-b\right)\mathcal{\widetilde{H}}_\gamma-\frac{3a}{4} + \frac{b}{2}\right)\\
&\hspace{5mm}+(-3)\left(-\frac{a}{2}+b\right)\left(-\frac{3}{4}c-\frac{2c-5}{2}\mathcal{\widetilde{H}}_\gamma' + \frac{17c-2}{12}\right)\\
%&=-6a +\frac{3b}{4} 
&=-\frac{3}{2}(2-5c)\left(-\frac{3}{2}a+\frac{b}{2}\right)+3\left(\frac{a}{2}-b\right)\left(\frac{2}{3}c-\frac{1}{6}\right)
\\&\hspace{5mm}+\frac{3}{2}(2-5c)\left(\frac{a}{2}-b\right)\left(\mathcal{\widetilde{H}}_\gamma+\mathcal{\widetilde{H}}_\gamma'\right).
\end{align*}
Finally using Lemma \ref{harmonic} we have
\[\widetilde{\mathcal{Y}}_1 + \widetilde{\mathcal{Y}}_2=\frac{1}{4}\left(a(17-41c)+b(-4+7c)\right).\]
In view of (\ref{gammap}) we have $\widetilde{\widehat{\mathcal{S}}}_p(a,b,c) =1/4(a(17-41c)+b(-4+7c))$ as required. 
%This completes the first part of the lemma.

\noindent
\textbf{Proof of $\widehat{\mathcal{T}}_n$}:
From Lemma \ref{sumident} we have $\mathcal{T}_n = (2n+2)!!-2(2n+1)!!$ and we obtain
\begin{align}
\widehat{\mathcal{T}}_n(a) &= (2n+3)!!\sum_{j=1}^{n-1}\frac{aj+1}{(2j+5)!!}(2j+3)((2j+2)!!-2(2j+1)!!)\notag\\
&=(2n+3)!!\sum_{j=1}^{n-1}\frac{(2j)!!}{(2j+5)!!}(aj+1)(2j+3)(2j+2)\notag\\
&\hspace{5mm}-2(2n+3)!!\sum_{j=1}^{n-1}\frac{aj+1}{2j+5}.
\end{align}
Let $p\geq7$ be a prime and put $n=p$. Then the reduction of the first sum is zero by Lemma \ref{polydoublefact}. The second sum has a simple pole at $j=(p-5)/2$. Hence, by Lemma \ref{pole} %$\widetilde{\widehat{\mathcal{T}}}_p$ is determined by 
\[\widehat{\mathcal{T}}_p(a) \equiv -\frac{2(2p+3)!!}{p}\cdot \frac{a(p-5)+2}{2} \mod p\Z_p.\] The corresponding reduction is then \[\widetilde{\widehat{\mathcal{T}}}_p(a) = (-2)(-3)\left(\frac{2-5a}{2}\right) = 6-15a.\] 
%We then note that for $p=3$ and $p=5$ we can compute the sum explicitly and we obtain 
%\begin{align*}\widehat{\mathcal{T}}_3(a) &= 9!!\sum_{j=1}^{2}\frac{aj+1}{(2j+5)!!}(2j+3)((2j+2)!!-2(2j+1)!!)\\
%&=9!!\left(\frac{a+1}{7!!}5(4!!-2(3!!)) + \frac{2a+1}{9!!}7(6!!-2(5!!))\right)
%&=216+342a,
%\end{align*}
%and
%\begin{align*}\widehat{\mathcal{T}}_5(a) &= 13!!\sum_{j=1}^{4}\frac{aj+1}{(2j+5)!!}(2j+3)((2j+2)!!-2(2j+1)!!)\\
%&=13!!\bigg(\frac{a+1}{7!!}5(4!!-2(3!!)) + \frac{2a+1}{9!!}7(6!!-2(5!!))\\
%&\hspace{4mm}+ \frac{3a+1}{11!!}9(8!!-2(7!!))+ \frac{4a+1}{13!!}11(10!!-2(9!!))\bigg)
%&=72696 + 195780a.
%\end{align*}
%So, $\widetilde{\widehat{\mathcal{T}}}_3 = 0$ and $\widetilde{\widehat{\mathcal{T}}}_5 = 1$. Thus, the lemma holds for all odd primes.

\noindent
\textbf{Proof of $\widehat{\mathcal{R}}_n$}:
From Lemma \ref{sumident} we have $\mathcal{R}_n = (2n+1)!!-(2n)!!$ and we obtain
\begin{align}
\widehat{\mathcal{R}}_n(a) &= (2n+3)!!\sum_{j=1}^{n-1}\frac{2j+2}{(2j+5)!!}(aj+1)((2j+1)!!-(2j)!!)\notag\\
&=(2n+3)!!\sum_{j=1}^{n-1}\frac{(2j+2)(aj+1)}{(2j+3)(2j+5)}\notag\\
&\hspace{5mm}-(2n+3)!!\sum_{j=1}^{n-1}\frac{(2j)!!}{(2j+5)!!}(aj+1)(2j+2).
\end{align}
Let $p\geq7$, and $n=p$. Then the reduction of the second sum is zero by Lemma \ref{polydoublefact}. The only poles of the first sum occur at $j=(p-5)/2$ and $j=(p-3)/2$. Hence, by Lemma \ref{pole}% $\widetilde{\widehat{\mathcal{R}}}_p$ is solely determined by the following expression
\[\widehat{\mathcal{R}}_p(a) \equiv \frac{(2p+3)!!}{p}\left( \frac{(p-3)(a(p-5)+2)}{2(p-2)} + \frac{(p-1)(a(p-3)+2)}{2(p+2)} \right)\mod p\Z_p.\] The corresponding reduction is then \[\widetilde{\widehat{\mathcal{R}}}_p(a) = (-3)\left(\frac{-12a+4}{4}\right) =9a-3.\]
%For $p=3$ and $p=5$ we compute $\mathcal{\widehat{R}}_p(a)$ explicitly and obtain 
%\begin{align*}\widehat{\mathcal{R}}_3(a) &= 9!!\sum_{j=1}^{2}\frac{2j+2}{(2j+5)!!}(aj+1)((2j+1)!!-(2j)!!)\\
%&=9!!\left(\frac{a+1}{7!!}4(3!!-2!!) + \frac{2a+1}{9!!}6(5!!-4!!)\right)
%&=78+120a,
%\end{align*}
%and
%\begin{align*}\widehat{\mathcal{R}}_5(a) &= 13!!\sum_{j=1}^{4}\frac{aj+1}{(2j+5)!!}(2j+2)((2j+1)!!-(2j)!!)\\
%&=13!!\bigg(\frac{a+1}{7!!}4(3!!-2!!) + \frac{2a+1}{9!!}6(5!!-4!!)\\
%&\hspace{4mm}+ \frac{3a+1}{11!!}8(7!!-6!!)+ \frac{4a+1}{13!!}10(9!!-8!!)\bigg)
%&=22692 + 57384a.
%\end{align*}
%So, $\widetilde{\widehat{\mathcal{R}}}_3 = 0$ and $\widetilde{\widehat{\mathcal{R}}}_5 = 4a+2$. Thus, the lemma holds for all odd primes.

\noindent
\textbf{Proof of $\mathcal{Z}_n$}:
Let $p\geq7$ and $n=p$. Then we note that $\mathcal{Z}_p$ has two simple poles occurring at $j=(p-3)/2$ and $j=(p-5)/2$. Hence, by Lemma \ref{pole} we have\[\mathcal{Z}_p(a,b,c) \equiv \frac{(2p+3)!!}{p}\left(\frac{a(p-3)^2+2b(p-3)+4c}{4(p+2)} + \frac{a(p-5)^2+2b(p-5)+4c}{4(p-2)}\right) \mod p\Z_p.\] The corresponding reduction is then \[\widetilde{\mathcal{Z}}_p(a,b,c) = (-3)\left(\frac{-16a+4b}{8}\right) = 6a-\frac{3}{2}b.\] This completes the proof of Lemma \ref{Ssum}.
%\noindent
%\textbf{Proof of $\psi_n$}:
%From the defintion we have 
%\begin{align*}
%\psi_n &= (2n+3)!!\sum_{j=1}^{n-1}\frac{j(2j+1)!!}{(2j+5)!!}\\
%&= (2n+3)!!\sum_{j=1}^{n-1}\frac{j}{(2j+3)(2j+5)}.
%\end{align*}
%Let $n=p$ we see that $\psi_p$ has simple poles for $j=(p-3)/2$ and $j=(p-5)/2$. Hence, \[\psi_p = \frac{(2p+3)!!}{p}\left(\frac{p-5}{2p-4} + \frac{p-3}{2p+4}\right) \mod p\Z_p.\] The corresponding reduction is then \[\widetilde{\psi}_p = (-3)\left(\frac{5}{4}-\frac{3}{4}\right) = -\frac{3}{2}.\] This completes the proof of the lemma.
\end{proof}

\begin{lemma}\label{uvxw} Let $p\geq7$ be a prime. For each integer $n\geq 1$ we define $\mathcal{\widehat{U}}_n, \mathcal{\widehat{V}}_n , \mathcal{\widehat{W}}_n$ and $\mathcal{\widehat{X}}_n$ in $\mathbb{Q}_p$ as
\begin{equation}\label{uhatt}\mathcal{\widehat{U}}_n(a,b) := (2n+3)!!\sum_{j=1}^{n-1}\frac{(2j+4)}{(2j+5)!!}\mathcal{U}_j(a,b)\end{equation}
\begin{equation}\label{vhatt}\mathcal{\widehat{V}}_n(a,b,c) :=(2n+3)!!\sum_{j=1}^{n-1}\frac{(2j+4)}{(2j+5)!!}\mathcal{V}_j(a,b,c)\end{equation}
\begin{equation}\label{whatt}\mathcal{\widehat{W}}_n :=(2n+3)!!\sum_{j=1}^{n-1}\frac{(2j+4)}{(2j+5)!!}\mathcal{W}_j,\end{equation}
and
\begin{equation}\label{xhatt}\mathcal{\widehat{X}}_n(a,b) := (2n+3)!!\sum_{j=1}^{n-1}\frac{(2j+4)}{(2j+5)!!}\mathcal{X}_j(a,b).\end{equation}
Then \[\widetilde{\mathcal{\widehat{U}}}_p(a,b) = a - \frac{19}{2}b, \quad \widetilde{\mathcal{\widehat{V}}}_p(a,b,c) = -\frac{45}{4}a + b +3c, \quad \widetilde{\mathcal{\widehat{W}}}_p = 2 \text{ and } \widetilde{\mathcal{\widehat{X}}}_p(a,b) = -3(a-b).\]
\end{lemma}

\begin{proof}
\textbf{Proof of $\widetilde{\mathcal{\widehat{U}}}_p$}:
From the definition of $\mathcal{U}_n$ in (\ref{Udef}) and $\mathcal{S}_n$ in (\ref{Sdef}) and by applying Lemma \ref{doublefactinner} we obtain
\begin{align}\label{uhattstart}\mathcal{\widehat{U}}_n(a,b) &= (2n+3)!!\sum_{j=1}^{n-1}\frac{(2j+4)!!}{(2j+5)!!}\sum_{i=1}^{j-1}\frac{(2j+3)\mathcal{S}_j(a,b)}{(2j+4)!!}\notag\\
&=(2n+4)!!\sum_{i=1}^{n-1}\frac{(2i+3)\mathcal{S}_i(a,b)}{(2i+4)!!} - (2n+3)!!\sum_{i=1}^{n-1}\frac{(2i+3)\mathcal{S}_i(a,b)}{(2i+4)!!}\cdot\frac{(2i+6)!!}{(2i+5)!!}\notag\\
&=(2n+4)!!\sum_{i=1}^{n-1}\frac{(2i+3)!!}{(2i+4)!!}\sum_{k=1}^i \frac{ak+b}{2k+1} - (2n+3)!!\sum_{i=1}^{n-1}\frac{(2i+6)}{(2i+5)}\sum_{k=1}^i \frac{ak+b}{2k+1}.
\end{align}
Let $S_1$ and $S_2$ denote the two sums in (\ref{uhattstart}). Let $p\geq7$ and put $n=p$. Then we note that $\nu_p(S_1)\geq 1$ for all $i$ and $k$ except for $i\in[p-2,p-1]$ and $k=(p-1)/2$. In view of Lemma \ref{pole} this yields, %$\widetilde{S}_1$ is solely determined by% its last two terms and when $k=(p-1)/2$. This yields
\[S_1 \equiv \frac{(2p+4)!!}{p}\left(\frac{(2p-1)!!}{(2p)!!}\left(a\left(\frac{p-1}{2}\right)+b\right) + \frac{(2p+1)!!}{(2p+2)!!}\left(a\left(\frac{p-1}{2}\right)+b\right)\right) \mod p\mathbb{Z}_p.\] Using Lemma \ref{generalwilson} we obtain \[\widetilde{S}_1 = -16\left(\frac{1}{2}(-a/2+b) +\frac{1}{4}(-a/2+b)\right) = 6a-12b.\]

Concerning $S_2$, we put
%it has a simple pole at $i=(p-5)/2$ and one for each $k\geq(p-1)/2$. This give rise to a similar situation as in Lemma \ref{Ssum}. Let 
\[\mathcal{Y}_1:=\frac{(2p+3)!!}{p}(p+1)\sum_{k=1}^{(p-5)/2} \frac{ak+b}{2k+1},\] and \[\mathcal{Y}_2 := \frac{(2p+3)!!}{p}\left(a\left(\frac{p-1}{2}\right)+b\right)\sum_{i=(p-1)/2}^{p-1} \frac{2i+6}{2i+5}.\] Then we note that $S_2 \equiv \mathcal{Y}_1 + \mathcal{Y}_2 \mod p\mathbb{Z}_p.$ Using (\ref{sum1done}) in the proof of Lemma \ref{Ssum} we have \[\sum_{k=1}^{(p-5)/2} \frac{ak+b}{2k+1} = \frac{a(p-3)}{4}-\left(\frac{a}{2}-b\right)\mathcal{H}_\gamma-\frac{a(p-3)+2b}{2p-4}.\] It follows that \begin{equation}\label{y1done22}\mathcal{Y}_1 = \frac{(2p+3)!!}{p}(p+1)\left( \frac{a(p-3)}{4}-\left(\frac{a}{2}-b\right)\mathcal{H}_\gamma-\frac{a(p-3)+2b}{2p-4}\right).\end{equation} By a switch of index in the sum in $\mathcal{Y}_2$ we obtain 
\begin{align}\label{summy}
\sum_{i=(p-1)/2}^{p-1} \frac{2i+6}{2i+5} &= \sum_{i=(p-1)/2}^{p-1} 1+\frac{1}{2i+5} \notag\\
&= \sum_{j=1}^{(p+1)/2}1+ \frac{1}{2j+p+2} \notag\\
&=\frac{p+1}{2} + \sum_{j=1}^{(p+1)/2}\frac{1}{2j+p+2}\notag\\
&=\frac{p+1}{2} -\frac{1}{p+2} + \sum_{j=1}^{(p+3)/2}\frac{1}{2j+p}.
\end{align}
By fixing the index in (\ref{summy}) to fit the conditions of Lemma \ref{harmonic} we obtain 
\begin{align}\sum_{i=(p-1)/2}^{p-1} \frac{2i+6}{2i+5} &=\frac{p+1}{2} -\frac{1}{p+2} + \sum_{j=1}^{(p+3)/2}\frac{1}{2j+p}\notag\\
&= \frac{p+1}{2} -\frac{1}{p+2}+\frac{1}{2p+3}+\frac{1}{2p+1}+\frac{1}{2p-1} + \sum_{j=1}^{(p-3)/2}\frac{1}{2j+p}.\notag
\end{align}
Put \[\sigma' := \frac{p+1}{2} -\frac{1}{p+2}+\frac{1}{2p+3}+\frac{1}{2p+1}+\frac{1}{2p-1}.\]
Then by definition we have \begin{equation}\label{y2done22}\mathcal{Y}_2 = \frac{(2p+3)!!}{p}\left(a\left(\frac{p-1}{2}\right)+b\right)\left(\sigma' +  \sum_{j=1}^{(p-3)/2}\frac{1}{2j+p}\right).\end{equation}
We note that $\widetilde{\sigma'} = 1/3$, and together with (\ref{y1done22}) and (\ref{y2done22}) we then have
\begin{align}
\widetilde{S}_2 &= \widetilde{\mathcal{Y}}_1 + \widetilde{\mathcal{Y}}_2\notag\\
&= (-3)\left(\frac{-3a}{4}-\left(\frac{a}{2}-b\right)\mathcal{\widetilde{H}}_\gamma + \frac{-3a+2b}{4}\right)\notag\\
&\hspace{5mm} + (-3)\left(-\frac{a}{2} +b\right)\left(\widetilde{\sigma'} + \widetilde{\mathcal{H}}_\gamma'\right)\notag\\
&=5a -\frac{5}{2}b + \left(\frac{3}{2}a-3b\right)\left(\mathcal{\widetilde{H}}_\gamma  + \mathcal{\widetilde{H}}_\gamma'\right)\notag\\
&=5a -\frac{5}{2}b. 
\end{align}
Thus, from (\ref{uhattstart}) we obtain \[\widetilde{\mathcal{\widehat{U}}}_p(a,b) = \widetilde{S}_1 - \widetilde{S}_2 = 6a-12b-5a+\frac{5}{2}b = a - \frac{19}{2}b.\]

\noindent
\textbf{Proof of $\widetilde{\mathcal{\widehat{V}}}_p$}:
We recall the definition of $\mathcal{V}_j(a,b,c)$ in (\ref{Vdef}) and then apply Lemma \ref{doublefactinner} to obtain
\begin{align}\label{vhattstart}\mathcal{\widehat{V}}_n(a,b,c) &= (2n+3)!!\sum_{j=1}^{n-1}\frac{(2j+4)!!}{(2j+5)!!}\sum_{i=1}^{j-1}\frac{\mathcal{R}_j(ai^2+bi+c)}{(2j+4)!!}\notag\\
&=(2n+4)!!\sum_{i=1}^{n-1}\frac{\mathcal{R}_i(ai^2+bi+c)}{(2i+4)!!} - (2n+3)!!\sum_{i=1}^{n-1}\frac{\mathcal{R}_i(ai^2+bi+c)}{(2i+4)!!}\cdot\frac{(2i+6)!!}{(2i+5)!!}.
\end{align}
We recall that $\mathcal{R}_n = (2n+1)!!-(2n)!!$. Insertion into (\ref{vhattstart}) yields
\begin{align*}\mathcal{\widehat{V}}_n(a,b,c) &= (2n+4)!!\sum_{i=1}^{n-1}\frac{((2i+1)!!-(2i)!!(ai^2+bi+c)}{(2i+4)!!} \\&\hspace{5mm}- (2n+3)!!\sum_{i=1}^{n-1}\frac{((2i+1)!!-(2i)!!)(ai^2+bi+c)}{(2i+4)!!}\cdot\frac{(2i+6)!!}{(2i+5)!!}\\
&=(2n+4)!!\sum_{i=1}^{n-1}\frac{(2i+1)!!(ai^2+bi+c)}{(2i+4)!!} -(2n+4)!!\sum_{i=1}^{n-1}\frac{ai^2+bi+c}{(2i+2)(2i+4)} \\
&\hspace{5mm}-(2n+3)!!\sum_{i=1}^{n-1}\frac{(ai^2+bi+c)(2i+6)}{(2i+3)(2i+5)}\\
&\hspace{5mm} + (2n+3)!!\sum_{i=1}^{n-1}\frac{(2i)!!(ai^2+bi+c)(2i+6)}{(2i+5)!!}
\end{align*}
We denote each individual sum in the above equation as $S_1, S_2, S_3$ and $S_4$ respectively. Let $p\geq7$ and put $n=p$. Then we note that $\nu_p\left(\frac{(2i+1)!!}{(2i+4)!!}\right)\geq0$, and thus the reduction of $S_1$ is 0. However $S_2$ and $S_3$ has simple poles occurring at $i = p-1$ and $i = p-2$ and $i = (p-3)/2$ and $i=(p-5)/2$ respectively. By Lemma \ref{pole} this yields \[S_2 \equiv \frac{(2p+4)!!}{4p}\left(\frac{a(p-2)^2 + b(p-2)+c}{p-1}+\frac{a(p-1)^2 + b(p-1)+c}{p+1}\right) \mod p\mathbb{Z}_p,\] and 
\[S_3 \equiv \frac{(2p+3)!!}{p}\left(\frac{\left(a\left(\frac{p-3}{2}\right)^2 + b\left(\frac{p-3}{2}\right) + c\right)(p+3)}{p+2} + \frac{\left(a\left(\frac{p-5}{2}\right)^2 + b\left(\frac{p-5}{2}\right) + c\right)(p+1)}{p-2}\right) \mod p\mathbb{Z}_p.\]
The corresponding reductions are \[\widetilde{S}_2 = -4(-3a+b) = 12a-4b,\] and
\[\widetilde{S}_3 = -3\left(\frac{1}{4}a-b+c\right) = -\frac{3}{4}a +3b-3c.\] Finally we have $\widetilde{S}_4 = 0$ by Lemma \ref{polydoublefact}. Hence, \[\widetilde{\mathcal{\widehat{V}}}_p(a,b,c) = -\left(\widetilde{S}_2 + \widetilde{S}_3\right) = -\frac{45}{4}a + b +3c.\]
%For $p=3$ and $p=5$ we compute $\widehat{\mathcal{V}}_p$ explicitly. We have 
%\begin{align*}
%\widehat{\mathcal{V}}_3(a,b,c) &= 9!!\sum_{j=1}^{2}\frac{(2j+4)!!}{(2j+5)!!}\sum_{i=1}^{j-1}\frac{\mathcal{R}_i(ai^2+bi+c)}{(2i+4)!!} \\
%&= 9!!\left(\frac{8!!}{9!!}\frac{\mathcal{R}_1(a+b+c)}{6!!}\right)\\
%&=8(3!!-2!!)(a+b+c)\\
%\end{align*}
%and we obtain $\widetilde{\widehat{\mathcal{V}}}_3(a,b,c) = 2(a+b+c)$. For $p=5$ we have
%\begin{align*}
%\widehat{\mathcal{V}}_5(a,b,c) &= 13!!\sum_{j=1}^{4}\frac{(2j+4)!!}{(2j+5)!!}\sum_{i=1}^{j-1}\frac{\mathcal{R}_i(ai^2+bi+c)}{(2i+4)!!} \\
%&= 13!!\bigg(\frac{8!!}{9!!}\frac{\mathcal{R}_1(a+b+c)}{6!!} + \frac{10!!}{11!!}\left(\frac{\mathcal{R}_1(a+b+c)}{6!!}+\frac{\mathcal{R}_2(4a+2b+c)}{8!!}\right) \\
%&\hspace{5mm}+ \frac{12!!}{13!!}\left(\frac{\mathcal{R}_1(a+b+c)}{6!!}+\frac{\mathcal{R}_2(4a+2b+c)}{8!!} + \frac{\mathcal{R}_3(9a+3b+c)}{10!!}\right) \bigg).
%\end{align*}
%All terms except the first and last contains a factor $5$ and therefore is zero in the reduction. Hence, \[\widetilde{\widehat{\mathcal{V}}}_5(a,b,c) = - (a+b+c) - (9a+3b+c) =b+3c.\]

\noindent

\textbf{Proof of $\widetilde{\mathcal{\widehat{W}}}_p$}:
From the definition of $\mathcal{W}_n$ in (\ref{Wdef}) and $\mathcal{T}_n$ in (\ref{Tdef}) and Lemma \ref{doublefactinner} we obtain
\begin{align}\label{whattstart}\mathcal{\widehat{W}}_n &= (2n+3)!!\sum_{j=1}^{n-1}\frac{(2j+4)!!}{(2j+5)!!}\sum_{i=1}^{j-1}\frac{(2i+3)\mathcal{T}_i}{(2i+4)!!}\notag\\
&=(2n+4)!!\sum_{i=1}^{n-1}\frac{(2i+3)\mathcal{T}_i}{(2i+4)!!} - (2n+3)!!\sum_{i=1}^{n-1}\frac{(2i+3)\mathcal{T}_i}{(2i+4)!!}\cdot\frac{(2i+6)!!}{(2i+5)!!}\notag\\
&=(2n+4)!!\sum_{i=1}^{n-1}\frac{(2i+3)((2i+2)!!-2(2i+1)!!)}{(2i+4)!!} \notag\\
&\hspace{5mm}-(2n+3)!!\sum_{i=1}^{n-1}\frac{(2i+3)((2i+2)!!-2(2i+1)!!)}{(2i+4)!!}\cdot\frac{(2i+6)!!}{(2i+5)!!}\notag\\
&=(2n+4)!!\sum_{i=1}^{n-1}\frac{2i+3}{2i+4} - 2(2n+4)!!\sum_{i=1}^{n-1}\frac{(2i+3)(2i+1)!!}{(2i+4)!!} \notag\\
&\hspace{5mm}-(2n+3)!!\sum_{i=1}^{n-1}(2i+3)(2i+6)(2i+2)\cdot\frac{(2i)!!}{(2i+5)!!}+2(2n+3)!!\sum_{i=1}^{n-1}\frac{2i+6}{2i+5}.
\end{align}
Let $p\geq7$ and put $n=p$, and then let each of the sums the last equality of (\ref{whattstart}) be denoted $S_1, S_2, S_3$ and $S_4$ respectively. First we note that $\widetilde{S}_3 = 0$ by Lemma \ref{polydoublefact}. Secondly, $S_2$ has no poles. Thus, $\widetilde{S}_2 = 0$. Hence, $\widetilde{\mathcal{\widehat{W}}}_p = \widetilde{S}_1 + \widetilde{S}_4$. Note that $S_1$ and $S_4$ have simple poles, occurring at $i=p-2$ and $i=(p-5)/2$ respectively. This yields
\[S_1 \equiv \frac{(2p+4)!!}{p}\cdot\frac{2p-1}{2} \mod p\mathbb{Z}_p,\] and \[S_4 \equiv \frac{2(2p+3)!!}{p}(p+1) \mod p\mathbb{Z}_p,\] from which we deduce that \begin{equation}\label{whattend}\widetilde{\mathcal{\widehat{W}}}_p = \widetilde{S}_1 + \widetilde{S}_4 = 8-6=2.\end{equation}
%For $p=3$ and $p=5$ we compute $\widehat{\mathcal{W}}_p$ explicitly. We have 
%\begin{align*}
%\widehat{\mathcal{W}}_3 &= 9!!\sum_{j=1}^{2}\frac{(2j+4)!!}{(2j+5)!!}\sum_{i=1}^{j-1}\frac{(2i+3)\mathcal{T}_i}{(2i+4)!!} \\
%&= 9!!\left(\frac{8!!}{9!!}\frac{5\mathcal{T}_1}{6!!}\right)\\
%&=8\cdot5(4!!-2\cdot3!!)\\
%&=80,
%\end{align*}
%and thus $\widetilde{\widehat{\mathcal{W}}}_3 = 2$. For $p=5$ we obtain
%\begin{align*}
%\widehat{\mathcal{W}}_5 &= 13!!\sum_{j=1}^{4}\frac{(2j+4)!!}{(2j+5)!!}\sum_{i=1}^{j-1}\frac{(2i+3)\mathcal{T}_i}{(2i+4)!!} \\
%&= 13!!\left(\frac{8!!}{9!!}\frac{5\mathcal{T}_1}{6!!} + \frac{10!!}{11!!}\left(\frac{5\mathcal{T}_1}{6!!}+\frac{7\mathcal{T}_2}{8!!}\right)+ \frac{12!!}{13!!}\left(\frac{5\mathcal{T}_1}{6!!}+\frac{7\mathcal{T}_2}{8!!} + \frac{9\mathcal{T}_3}{10!!}\right) \right).
%\end{align*}
%The last term is the only non-zero term in the reduction. Hence, \[\widetilde{\widehat{\mathcal{W}}}_5 = 12\cdot9 \cdot\mathcal{T}_3 = 2.\]

\noindent
\textbf{Proof of $\widetilde{\mathcal{\widehat{X}}}_p$}:
Using the definition of $\mathcal{X}_n$ in (\ref{Xdef}) and Lemma \ref{doublefactinner} we have
\begin{align}\label{xhattstart}\mathcal{\widehat{X}}_n(a,b) &= (2n+3)!!\sum_{j=1}^{n-1}\frac{(2j+4)!!}{(2j+5)!!}\sum_{i=1}^{j-1}\frac{(ai+b)(2i+1)!!}{(2i+4)!!}\notag\\
&=(2n+4)!!\sum_{i=1}^{n-1}\frac{(ai+b)(2i+1)!!}{(2i+4)!!} - (2n+3)!!\sum_{i=1}^{n-1}\frac{(ai+b)(2i+1)!!}{(2i+4)!!}\cdot\frac{(2i+6)!!}{(2i+5)!!}\notag\\
&=(2n+4)!!\sum_{i=1}^{n-1}\frac{(ai+b)(2i+1)!!}{(2i+4)!!} - (2n+3)!!\sum_{i=1}^{n-1}\frac{(ai+b)(2i+6)}{(2i+3)(2i+5)}.
\end{align}
Let $p\geq7$ and put $n=p$. The first sum in (\ref{xhattstart}) contains no poles and therefore has reduction zero. The second sum has simple poles at $i = (p-3)/2$ and $i=(p-5)/2$. Then by Lemma \ref{pole}
\[\mathcal{\widehat{X}}_p(a,b) \equiv -\frac{(2p+3)!!}{p}\left(\frac{\left(a\left(\frac{p-3}{2}\right) + b\right)(p+3)}{p+2} + \frac{\left(a\left(\frac{p-5}{2}\right) + b\right)(p+1)}{p-2}\right) \mod p\mathbb{Z}_p.\] This corresponds to the reduction 
\[\mathcal{\widetilde{\widehat{X}}}_p(a,b) = -3(a-b).\] 
%For $p=3$ and $p=5$ we compute $\widehat{\mathcal{X}}_p$ explicitly. We have 
%\begin{align*}
%\widehat{\mathcal{X}}_3 &= 9!!\sum_{j=1}^{2}\frac{(2j+4)!!}{(2j+5)!!}\sum_{i=1}^{j-1}\frac{(ai+b)(2i+1)!!}{(2i+4)!!} \\
%&= 9!!\left(\frac{8!!}{9!!}\frac{5!!}{6!!}(a+b)\right)\\
%&=8\cdot5!!(a+b),
%\end{align*}
%and thus $\widetilde{\widehat{\mathcal{X}}}_3 = 0$. For $p=5$ we obtain
%\begin{align*}
%\widehat{\mathcal{X}}_5 &= 13!!\sum_{j=1}^{4}\frac{(2j+4)!!}{(2j+5)!!}\sum_{i=1}^{j-1}\frac{(ai+b)(2i+1)!!}{(2i+4)!!} \\
%&= 13!!\bigg(\frac{8!!}{9!!}\frac{3!!}{6!!}(a+b) + \frac{10!!}{11!!}\left(\frac{3!!}{6!!}(a+b)+\frac{5!!}{8!!}(2a+b)\right)\\
%&\hspace{5mm}+ \frac{12!!}{13!!}\left(\frac{3!!}{6!!}(a+b)+\frac{5!!}{8!!}(2a+b) + \frac{7!!}{10!!}(3a+b)\right) \bigg)\\
%\end{align*}
%The last term is the only non-zero term in the reduction. Hence, \[\widetilde{\widehat{\mathcal{W}}}_5 = 13\cdot11 \cdot8\cdot3(a+b) = -3(a+b).\]
%This completes the proof of Lemma \ref{uvxw}.
\end{proof}

An important part of the proof of the Main Lemma is our ability to explicitly express the solutions of the difference equations, and thus the following lemma from \cite{Elaydi2005} is of importance.

\begin{lemma}\cite[\S 1.2]{Elaydi2005}\label{thmDIFF}
Let $k$ be a field. Furthermore let $f,g : \mathbb{Z}^+ \rightarrow k$, and $y_0 \in k$. Given a nonhomogeneous difference equation \[y_{n+1} = f(n)y_{n} + g(n), y_{n_0} = y_0\] where $n_0 \in [0,n]$. The general solution to the difference equation is given by 
\begin{equation}\label{elaydieq}
y_{n} = \left[\prod_{j=n_0}^{n-1}f(j)\right]y_0 + \sum_{r = n_0}^{n-1}\left[\prod_{j=r+1}^{n-1}f(j)\right]g(r)
.\end{equation}
\end{lemma}

%Now we are set to prove our Main Lemma and Theorem \ref{pncoeff}.%, which is done in the upcoming section.

\section{ Main Lemma and proof of Theorem \ref{pncoeff} }\label{secproof}
In this section we state and prove our Main Lemma, and give the proof of Theorem \ref{pncoeff}. For convenience we give the following definition. Given a ring $R$ and an element $a\in R$ we let $\langle a \rangle$ denote the ideal of $R$ generated by $a$.
\begin{namedlemma}[Main Lemma]\label{prop2ramif}
Let $p$ be an odd prime and let $k$ be a field of characteristic $p$. Let $f \in k[[\zeta]]$ be defined as \[f(\zeta) = \zeta\left(1 + \sum_{j=2}^{+\infty}a_j\zeta^j\right).\] %Furthermore let $\alpha_n, \beta_n$ and $\gamma_n$ be defined as in Theorem \ref{pncoeff}. 
Let $\alpha_1, \beta_1$ and $\gamma_1$ be defined as follows 
\begin{equation*}\label{alphabetagmma}
\alpha_1 := a_2^{p-2}\left(\frac{3}{2}a_2^3+a_3^2-a_2a_4\right), \quad
\beta_1 := \frac{a_3}{a_2}\alpha_1, \quad
\gamma_1 := -\left(\frac{3a_2}{2}-\frac{a_4}{a_2}\right)\alpha_1.
\end{equation*}
Then
\[f^p(\zeta) \equiv \zeta + \alpha_1 \zeta^{2p+3} +\beta_1 \zeta^{2p+4} +\gamma_1 \zeta^{2p+5} \mod \langle \zeta^{2p+6} \rangle.\]
\end{namedlemma}

%We start by proving the Main Lemma which will be used as the base step in the inductive proof of Theorem \ref{pncoeff}. 

\begin{proof}[Proof of Main Lemma]
Analogous to \cite{LindahlRiveraLetelier2013} for $m= 1$ we define the recurrence relation $\Delta_1(\zeta) := f(\zeta) - \zeta$ and for $m \geq 2$ \[\Delta_m(\zeta) := \Delta_{m-1}(f(\zeta)) - \Delta_{m-1}(\zeta).\] Note that $\Delta_p(\zeta) = f^p(\zeta) - \zeta$, and $\ord(\Delta_m) \geq 2m+1$.% We want to find the coefficients of the terms of degree less than or equal to $2p+5$ of $f^p$. Thus, we need to keep track of the five lowest degree terms of $\Delta_m$ for each $m$. %This also implies that the terms of degree $\geq 8$ of $f$ won't affect our sought terms and thus can be ignored in the proof.

For technical reasons we define $F_\ell := \mathbb{Q}_p[x_1,x_2,\dots,x_\ell]$, and $F_\infty := \mathbb{Q}_p[x_1,x_2,\dots]$. Moreover we consider the power series $\widehat{f} \in F_\infty[[\zeta]]$ defined as \[\widehat{f}(\zeta) := \zeta\left(1 + \sum_{j=1}^{+\infty}x_j\zeta^{j+1}\right).\] 
%i.e. such that for each integer $i\geq 1$ the evaluation mapping $x_i \mapsto a_i$ maps $\widehat{g}$ to $g$. 
For $m=1$ we define the relation $\widehat{\Delta}_1(\zeta) := \widehat{f}(\zeta)-\zeta$ and for each integer $m\geq 2$ \[\widehat{\Delta}_m(\zeta) := \widehat{\Delta}_{m-1}(\widehat{f}(\zeta))-\widehat{\Delta}_{m-1}(\zeta).\]  
Defined in this way there is a clear relation between $f(\zeta)$ and $\widehat{f}(\zeta)$ and thus  between $\Delta_m$ and $\widehat{\Delta}_m$. In the last part of the proof we exploit this relation to find the coefficients of $f^p(\zeta) - \zeta$.

Let $x =x_\infty := (x_1,x_2,\dots)$, and let $\mathcal{A}_m(x), \mathcal{B}_m(x), \mathcal{C}_m(x), \mathcal{D}_m(x)$ and $\mathcal{E}_m(x)$ be  defined by 
\begin{align}
\label{Adiff}\mathcal{A}_{m+1}(x) &= x_1(2m+1)\mathcal{A}_m,\\
\label{Bdiff}\mathcal{B}_{m+1}(x) &= x_2(2m+1)\mathcal{A}_m+ x_1(2m+2)\mathcal{B}_m,\\
\label{Cdiff}\mathcal{C}_{m+1}(x) &= \left(x_1^2m(2m+1)+x_3(2m+1)\right)\mathcal{A}_m+ x_2(2m+2)\mathcal{B}_m + x_1(2m+3)\mathcal{C}_m,\\
\label{Ddiff}\mathcal{D}_{m+1}(x) &= (x_4(2m+1) + x_1x_2(2m+1)(2m))\mathcal{A}_m+ (x_3(2m+2)+x_1^2(m+1)(2m+1))\mathcal{B}_m \\&\hspace{5mm}+ x_2(2m+3)\mathcal{C}_m + x_1(2m+4)\mathcal{D}_m\notag\end{align}
and
\begin{align}\label{Ediff}\mathcal{E}_{m+1}(x) &= \left(x_5(2m+1) + x_1x_3(2m+1)(2m) + x_1^3\binom{2m+1}{3} + m(2m+1)x_2^2\right)\mathcal{A}_m\\\notag
&\hspace{5mm}+ (x_4(2m+2) + x_1x_2(2m+2)(2m+1))\mathcal{B}_m \notag\\
&\hspace{5mm}+ \left(x_3(2m+3)+x_1^2\binom{2m+3}{2}\right)\mathcal{C}_m + x_2(2m+4)\mathcal{D}_m + x_1(2m+5)\mathcal{E}_m,\notag
\end{align}
with initial conditions $(\mathcal{A}_1, \mathcal{B}_1, \mathcal{C}_1, \mathcal{D}_1,\mathcal{E}_1) = (x_1,x_2,x_3,x_4,x_5)$. 
We prove by induction that for $m\geq1$ we have \begin{equation}\label{deltapnstart}\widehat{\Delta}_m(\zeta) \equiv \mathcal{A}_m(x)\zeta^{2m+1}+ \mathcal{B}_m(x)\zeta^{2m+2}+ \mathcal{C}_m(x)\zeta^{2m+3}+ \mathcal{D}_m(x)\zeta^{2m+4}+ \mathcal{E}_m(x)\zeta^{2m+5}\mod \langle \zeta^{2m+6}\rangle.\end{equation} Throughout the rest of the proof let $\mathcal{A}_m:=\mathcal{A}_m(x) ,  \mathcal{B}_m:=\mathcal{B}_m(x),   \mathcal{C}_m:=\mathcal{C}_m(x), \mathcal{D}_m:=\mathcal{D}_m(x)$ and $\mathcal{E}_m:=\mathcal{E}_m(x)$ unless otherwise specified. 

For $m=1$ (\ref{deltapnstart}) holds by definition. Let $m\geq1$ be such that (\ref{deltapnstart}) holds. Then
\begin{align*}
\widehat{\Delta}_{m+1}(\zeta) &\equiv \widehat{\Delta}_{m}(\widehat{f}(\zeta)) - \widehat{\Delta}_m(\zeta)\\
&\equiv\mathcal{A}_m[(\zeta + \widehat{\Delta}_1(\zeta))^{2m+1}-\zeta^{2m+1}] + \mathcal{B}_m[(\zeta + \widehat{\Delta}_1(\zeta))^{2m+2}-\zeta^{2m+2}]\\
&\hspace{5mm}+\mathcal{C}_m[(\zeta + \widehat{\Delta}_1(\zeta))^{2m+3}-\zeta^{2m+3}] + \mathcal{D}_m[(\zeta + \widehat{\Delta}_1(\zeta))^{2m+4}-\zeta^{2m+4}]\\
&\hspace{5mm}+\mathcal{E}_m[(\zeta + \widehat{\Delta}_1(\zeta))^{2m+5}-\zeta^{2m+5}]\\
&\equiv\mathcal{A}_m\bigg[x_1(2m+1)\zeta^{2m+3} + x_2(2m+1)\zeta^{2m+4} + \left(x_3(2m+1)+x_1^2\binom{2m+1}{2}\right)\zeta^{2m+5} \\
&\hspace{5mm}+ (x_4(2m+1)+x_1x_2(2m+1)(2m))\zeta^{2m+6} \\
&\hspace{5mm}+ \left(x_5(2m+1) + x_1x_3(2m+1)(2m) + x_1^3\binom{2m+1}{3} + x_2^2\binom{2m+1}{2}\right)\zeta^{2m+7}\bigg]\\
&\hspace{5mm}+\mathcal{B}_m\bigg[x_1(2m+2)\zeta^{2m+4} + x_2(2m+2)\zeta^{2m+5} + \left(x_1^2\binom{2m+2}{2} + x_3(2m+2)\right)\zeta^{2m+6}\\
&\hspace{5mm} + (x_4(2m+2) + x_1x_2(2m+2)(2m+1))\zeta^{2m+7}\bigg]\\
&\hspace{5mm}+\mathcal{C}_m\bigg[x_1(2m+3)\zeta^{2m+5} + x_2(2m+3)\zeta^{2m+6} + \left(x_3(2m+3)+x_1^2\binom{2m+3}{2}\right)\zeta^{2m+7}\bigg]\\
&\hspace{5mm}+\mathcal{D}_m\bigg[x_1(2m+4)\zeta^{2m+6} + x_2(2m+4)\zeta^{2m+7}\bigg]+\mathcal{E}_m\bigg[x_1(2m+5)\zeta^{2m+7}\bigg]\\
&\equiv \mathcal{A}_{m+1}\zeta^{2m+3}+\mathcal{B}_{m+1}\zeta^{2m+4}+\mathcal{C}_{m+1}\zeta^{2m+5}+\mathcal{D}_{m+1}\zeta^{2m+6}+\mathcal{E}_{m+1}\zeta^{2m+7} \mod \langle \zeta^{2m+8} \rangle.
\end{align*}
This completes the proof of the induction step and proves (\ref{deltapnstart}), and we note that the difference equations only depend on the coefficients $x_1,x_2,x_3,x_4,x_5$ and not those for higher order terms, so $\mathcal{A}_m,\mathcal{B}_m,\mathcal{C}_m,\mathcal{D}_m,\mathcal{E}_m \in F_5$.

We divide the proof of the Lemma in three cases. First we consider the cases $p=3$ and $p=5$. In the third case we prove the Lemma for all $p\geq7$.

For the cases $p=3$ and $p=5$ we will explicitly compute the values of the coefficients in $\widehat{\Delta}_m$ for $m\in[1,3]$ and $m\in[1,5]$ respectively. For convenience  we redefine $F_5$ as $\mathbb{F}_p[x_1,\dots,x_5]$ for $p=3$ and $p=5$ respectively.\footnote{Note that $3\mid (2m+1)(2m)(2m-1)$ so $\binom{2m+1}{3}$ defines an element in $\F_3$.} Thereby we also utilize the fact that we are working over a field of characteristic 3 and 5 respectively. This fact will be used continuously without comment throughout the proof.

\partn{Case 1, $p=3$} From (\ref{Adiff}), (\ref{Bdiff}), (\ref{Cdiff}), (\ref{Ddiff}) and (\ref{Ediff}) we have $\mathcal{A}_2 = 0, \mathcal{A}_3 = 0$, and $\mathcal{B}_2 = x_1x_2, \mathcal{B}_3 = 0$. Furthermore, $\mathcal{C}_2 = x_2^2 + 2x_1x_3,  \mathcal{C}_3 = x_1(x_2^2+2x_1x_3) = x_1(x_2^2-x_1x_3)$. We also have $\mathcal{D}_2 = x_2x_3 + 2x_2x_3 = 0, \mathcal{D}_3 = x_2(x_2^2 -x_1x_3)$. Finally we obtain $\mathcal{E}_2 = x_1x_5+x_3(2x_3 + x_1^2)+x_2x_4,$ and  $\mathcal{E}_3 = x_3(x_2^2 - x_1x_3)$. Thus we have \[\widehat{\Delta}_3(\zeta) \equiv x_1(x_2^2-x_1x_3)\zeta^9 + x_2(x_2^2-x_1x_3)\zeta^{10} + x_3(x_2^2-x_1x_3)\zeta^{11} \mod \langle \zeta^{12} \rangle.\]
For each $i \in \{1,2,3\}$ we specialize each $x_i$ to $a_{i+1}$  and obtain \begin{align*}\widehat{\Delta}_3(\zeta) \equiv f^3(\zeta)-\zeta  &\equiv a_2(a_3^2-a_2a_4) \zeta^9 + a_3(a_3^2-a_2a_4) \zeta^{10} + a_4(a_3^2-a_2a_4) \zeta^{11} \\
&\equiv \alpha_1\zeta^9 +\beta_1\zeta^{10} + \gamma_1\zeta^{11}\mod \langle \zeta^{12} \rangle.\end{align*} This completes the proof of the Lemma for the case $p=3$.

%We explicitly compute $\Delta_3$ and obtain
%\begin{align*}
%\Delta_2(\zeta) &\equiv \Delta_1(f(\zeta)) - \Delta_1(\zeta)\\
%&\equiv a_2[(\zeta+\Delta_1)^{3}-\zeta^{3}]+ a_3[(\zeta+\Delta_1)^{4}-\zeta^{4}]\\
%&\hspace{5mm}+a_4[(\zeta+\Delta_1)^{5}-\zeta^{5}]+a_5[(\zeta+\Delta_1)^{6}-\zeta^{6}]\\
%&\hspace{5mm}+a_6[(\zeta+\Delta_1)^{7}-\zeta^{7}]\\
%&\equiv3a_2(\dots) + 4a_2a_3\zeta^{6} + 4a_3^2\zeta^{7} + 4a_3a_4\zeta^{8} + 4a_3a_5\zeta^{9}\\
%&\hspace{5mm}+5a_2a_4\zeta^{7} + 5a_3a_4\zeta^{8} + 5a_4^2\zeta^{9} + 6a_5(\dots) + 7a_2a_6\zeta^{9}\\
%&\equiv a_2a_3\zeta^{6} + (a_3^2 + 2a_2a_4)\zeta^{7} + (a_2a_6 + 2a_4^2 +a_3a_5)\zeta^{9}\mod \langle \zeta^{10} \rangle.
%\end{align*}
%Also,
%\begin{align*}
%\Delta_3(\zeta) &\equiv \Delta_2(f(\zeta)) - \Delta_2(\zeta)\\
%&\equiv a_2a_3[(\zeta+\Delta_1)^{6}-\zeta^{6}]+ (a_3^2+2a_2a_4)[(\zeta+\Delta_1)^{7}-\zeta^{7}]\\
%&\hspace{5mm}+(a_2a_6 + 2a_4^2 + a_3a_5)[(\zeta+\Delta_1)^{9}-\zeta^{9}]\\
%&6a_2a_3(\dots) + 7a_2(a_3^2+2a_2a_4)\zeta^{9} + 7a_3(a_3^2+2a_2a_4)\zeta^{10}\\
%&\hspace{5mm}+7a_4(a_3^2+2a_2a_4)\zeta^{11} + 9(a_2a_6 + 2a_4^2 + a_3a_5)(\dots)\\
%&\equiv a_2(a_3^2-a_2a_4)\zeta^{9} + a_3(a_3^2-a_2a_4)\zeta^{10} + a_4(a_3^2-a_2a_4)\zeta^{11}\mod \langle \zeta^{12} \rangle
%\end{align*}
%which proves the case $p=3$.

\partn{Case 2, $p=5$} We continue in a similar procedure as for $p = 3$ using (\ref{Adiff}), (\ref{Bdiff}), (\ref{Cdiff}), (\ref{Ddiff}) and (\ref{Ediff}) to compute $\widehat{\Delta}_5$. We have
\[\mathcal{A}_2 = 3x_1^2,\quad \mathcal{A}_3= 0,\quad \mathcal{A}_4 = 0,\quad\mathcal{A}_5 = 0,\] and \[\mathcal{B}_2 = 2x_1x_2,\quad \mathcal{B}_3 = 2x_1^2x_2,\quad \mathcal{B}_4 = x_1^3x_2,\quad \mathcal{B}_5 = 0.\] Concerning $\mathcal{C}_m$ we have
\[\mathcal{C}_2 = (3x_1^2+3x_3)\mathcal{A}_1 + 4x_2\mathcal{B}_1 = 3x_1^3 + 4x_2^2 + 3x_1x_3,\quad \mathcal{C}_3 = x_2\mathcal{B}_2 + 2x_1\mathcal{C}_2 = x_1^4+x_1^2x_3.\] Using $\mathcal{A}_3 = \mathcal{A}_4 = 0$ we obtain \[ \mathcal{C}_4 = 8x_2\mathcal{B}_3 + 9x_1\mathcal{C}_3 = x_1^2x_2^2+4x_1(x_1^4+x_1^2x_3),\quad \mathcal{C}_5 = 10x_2\mathcal{B}_4 + 11x_1\mathcal{C}_4 = x_1^3(4x_1^3 + x_2^2+4x_1x_3).\]
Continuing for $\mathcal{D}_m$ we have \[\mathcal{D}_2 = (3x_4+6x_1x_2)x_1 + (4x_3+6x_1^2)x_2 + 5(\cdots)\mathcal{C}_1 + 6x_1x_4 = 4x_1x_4 +2x_1^2x_2+4x_2x_3.\] %We will also see that we do not need to compute $\mathcal{D}_3$ since we have that $\mathcal{D}_4 = 0$. 
Again using $\mathcal{A}_3 = 0$ we obtain \[\mathcal{D}_4 = (8x_3+28x_1^2)\mathcal{B}_3 + 9x_2\mathcal{C}_3 + 10(\cdots)\mathcal{D}_3 = 2x_1^2x_2(8x_3+28x_1^2) + 9x_2(x_1^4+x_1^2x_3) = 0.\] Finally we have \[\mathcal{D}_5 = 5(\cdots)\mathcal{B}_4 + 11x_2\mathcal{C}_4 = x_2(x_1^2x_2^2+4x_1(x_1^4+x_1^2x_3)) = x_1^2x_2(4x_1^3+x_2^2+4x_1x_3).\]
Consequently, for $\mathcal{E}_m$ we have
\[\mathcal{E}_2 = (3x_5+6x_1x_3+x_1^3+3x_2)x_1+(4x_4+12x_1x_2)x_2 + 5(\cdots)x_3 + 6x_2x_4 + 7x_1x_5 = x_1^4+x_1^2x_3,\] and 
\begin{align*} \mathcal{E}_3 &=5(\cdots)\mathcal{A}_2 + 6x_2\mathcal{B}_2 + (7x_3 + 21x_1^2)\mathcal{C}_2 + 8x_2\mathcal{D}_2 + 9x_1\mathcal{E}_2\\
&=12x_1x_2x_4 + (2x_3 + x_1^2)(3x_1^3+4x_2^2+3x_1x_3) \\&\hspace{5mm}+ 3x_2(4x_1x_4+2x_1^2x_2+4x_2x_3) + 4x_1(x_1^4+x_1^2x_3)\\
&=4x_1x_2x_4 + 3x_1^3x_3+2x_1^5+x_1x_3^2.
\end{align*}
For $\mathcal{E}_4$ we have
\begin{align*}
\mathcal{E}_4 &= (8x_4+56x_1x_2)\mathcal{B}_3 + (9x_3+36x_1^2)\mathcal{C}_3 + 10(\cdots)\mathcal{D}_3 + 11x_1\mathcal{E}_3\\
&=(3x_4+x_1x_2)(2x_1^2x_2)+(4x_3+x_1^2)(x_1^4+x_1^2x_3) + x_1(4x_1x_2x_4 + 3x_1^3x_3+2x_1^5+x_1x_3^2)\\
&=2x_1^3x_2^2 + 3x_1^4x_3+3x_1^6.
\end{align*}
Finally $\mathcal{E}_5$
\begin{align*}
\mathcal{E}_5 &= 11x_3\mathcal{C}_4 + 13x_1\mathcal{E}_4\\
&=x_3(x_1^2x_2^2+4x_1(x_1^4+x_1^2x_3)) + 3x_1(2x_1^3x_2^2 + 3x_1^4x_3+3x_1^6)\\
&=4x_1^7+x_1^4x_2^2+3x_1^5x_3+x_1^2x_2^2x_3+4x_1^3x_3^2\\
&=x_1^3\left(x_1+\frac{x_3}{x_1}\right)(4x_1^3+x_2^2+4x_1x_3).
\end{align*}
Note that $-3/2\equiv 1 \pmod{5}$ and thus $(4x_1^3+x_2^2+4x_1x_3) = \left(\frac{3}{2}x_1^3+x_2^2 - x_1x_3\right)$. The proof is completed by specializing for each $i \in \{1,2,3\}$ $x_i$ to $a_{i+1}$, \[\widetilde{\mathcal{C}}_5(a_2,a_3,a_4) = \alpha_1,\]\[ \widetilde{\mathcal{D}}_5(a_2,a_3,a_4) = \frac{a_3}{a_2}\alpha_1 = \beta_1,\] and \[ \widetilde{\mathcal{E}}_5(a_2,a_3,a_4) = \left(a_2 + \frac{a_4}{a_2}\right)\alpha_1 = -\left(\frac{3a_2}{2} - \frac{a_4}{a_2}\right)\alpha_1 =  \gamma_1.\]
This completes the proof of the Lemma for the case $p=5$.

\partn{Case 3, $p\geq7$} 
Note for $p\geq7$ all lemmas stated in \textsection\ref{techres} apply. 

%\begin{equation}\label{matrix2}
%\begin{bmatrix}
  %x_1(2m + 1) & 0 & 0 & 0 & 0\\
 % x_2(2m + 1) & x_1(2m+2) & 0 &0 & 0\\
 % x_1^2m(2m + 1) & x_2(2m+2) & x_1(2m + 3) &0 & 0\\
  % x_1x_2(m+1)(2m + 1) & x_2^2m(2m+1) & x_2(2m + 3) & x_1(2m+4) & 0\\
   %    \binom{2m+1}{3}x_1^3 + m(2m+1)x_2^2 & x_1x_2(m+1)(2m+1) & \binom{2m+1}{3}x_1^2 & x_2(2m+4) & x_1(2m+5)\\
 %\end{bmatrix}
 %\begin{bmatrix}
 % \widehat{A}_{m} \\
 %\widehat{B}_{m} \\
 %\widehat{C}_{m} \\
 % \widehat{D}_{m} \\
  % \widehat{E}_{m} 
 %\end{bmatrix}
% =
 %\begin{bmatrix}
% \widehat{A}_{m+1} \\
 %\widehat{B}_{m+1} \\
 %\widehat{C}_{m+1} \\
 % \widehat{D}_{m+1} \\
  % \widehat{E}_{m+1} \
 %\end{bmatrix},
% \end{equation}
The equations (\ref{Adiff}), (\ref{Bdiff}) and (\ref{Cdiff}) were explicitly solved in \cite[page 267ff]{Fransson2017}:
\begin{equation}\label{Asolve}\mathcal{A}_m = x_1^m(2m-1)!!,\end{equation}
\begin{equation}\label{Bsolve}
\mathcal{B}_{m} = x_1^{m-1}x_2\mathcal{R}_m,
 \end{equation}
and
\begin{equation} \label{Csolve}\mathcal{C}_m = x_1^{m+1}\mathcal{S}_m(1,-1)  +x_1^{m-1}x_3\mathcal{S}_m(0,1)+ x_1^{m-2}x_2^2\Big(\mathcal{S}_m(2,0) - \mathcal{T}_m\Big).
\end{equation} 

The remaining part of the proof is to explicitly calculate $\mathcal{D}_m$ and $\mathcal{E}_m$ and compute the corresponding reductions $\widetilde{\mathcal{D}}_p$ and $\widetilde{\mathcal{E}}_p$.
Insertion of (\ref{Asolve}), (\ref{Bsolve}) and (\ref{Csolve}) into (\ref{Ddiff}) yields
\begin{align*}
\mathcal{D}_{m+1} &= (x_4 + x_1x_2(2m))x_1^{m}(2m+1)!! \\&+ x_1^{m-1}x_2\mathcal{R}_m(x_3(2m+2)+x_1^2(m+1)(2m+1))  \\&+x_2(2m+3)(x_1^{m+1}\mathcal{S}_m(1,-1) +x_1x_3\mathcal{S}_m(0,1) + x_1^{m-2}x_2^2(\mathcal{S}_m(2,0)-\mathcal{T}_m)) \\&+ x_1(2m+4)\mathcal{D}_m.
\end{align*}
Put $\phi_m := x_4(2m+2)!!$. Then $\widetilde{\phi}_p =0$. %By use of the substitution $\mathcal{D}_m^* = \frac{\mathcal{D}_m}{x_1^{m-3}(2m+2)!!}$ and 
%\begin{align*}
%\mathcal{D}_{m+1}^* &= x_1^{3}x_2(2m)(2m+1)!!/(2m+4)!! \\&+ x_1^3x_2\mathcal{R}_m(m+1)(2m+1)/(2m+4)!!  \\&+x_2(2m+3)(x_1^{3}\mathcal{S}_m(1,-1) + x_2^2(\mathcal{S}_m(2,0)-\mathcal{T}_m))/(2m+4)!! \\&+\mathcal{D}_m^*.
%\end{align*}
Lemma \ref{thmDIFF} yields %with $\mathcal{D}_1^* =\frac{x_1^2x_4}{4!!}$ we obtain the explicit solution 
\begin{align*}
\mathcal{D}_{m} &=\phi_m +  x_1^{m-3}(2m+2)!!\sum_{j=1}^{m-1}\frac{1}{(2j+4)!!}\bigg((x_1^2x_4 + x_1^{3}x_2(2j))(2j+1)!! \\&+ (x_1^2(j+1)(2j+1) + x_3(2j+2))x_1x_2\mathcal{R}_j \\
& +x_2(2j+3)(x_1^{3}\mathcal{S}_j(1,-1) +x_1x_3\mathcal{S}_j(0,1)+ x_2^2(\mathcal{S}_j(2,0)-\mathcal{T}_j))\bigg).
\end{align*}
In view of the definitions in Lemma \ref{dmcoeff} we have
\begin{align}\label{fourtheq}\mathcal{D}_{m} &=  \phi_m + x_1^{m}x_2(\mathcal{X}_m(2,0)+\mathcal{U}_m(1,-1) + \mathcal{V}_m(2,3,1)) 
\\&\hspace{5mm}+ x_1^{m-3}x_2^3(\mathcal{U}_m(2,0)-\mathcal{W}_m)\notag\\
&\hspace{5mm}+x_1^{m-1}x_4\mathcal{X}_m(1,0)\notag\\
&\hspace{5mm}+x_1^{m-2}x_2x_3(\mathcal{U}_m(0,1)+\mathcal{V}_m(0,2,2))\notag
\end{align}
Inserting the explicit solution for $\mathcal{D}_m$ in (\ref{Ediff}) yields
\begin{align*}
\mathcal{E}_{m+1} &= x_1(2m+5)\mathcal{E}_m\\
&\hspace{1mm}+x_1^m(2m-1)!!\bigg(x_5(2m+1) + x_1x_3(2m+1)(2m) + x_1^3\binom{2m+1}{3}+ x_2^2\binom{2m+1}{2}\bigg)\\
&\hspace{1mm}+x_1^{m-1}x_2\mathcal{R}_m(x_1x_2(2m+2)(2m+1) + x_4(2m+2))\\
&\hspace{1mm}+x_1^{m-2}(2m+3)(x_3 + x_1^2(m+1))(x_1^3\mathcal{S}_m(1,-1)+x_1x_3\mathcal{S}_m(0,1)+x_2^2(\mathcal{S}_m(2,0)-\mathcal{T}_m))\\
&\hspace{1mm}+x_2(2m+4)\bigg(x_1^{m}x_2(\mathcal{X}_m(2,0)+\mathcal{U}_m(1,-1) + \mathcal{V}_m(2,3,1)) 
\\
&\hspace{1mm}+ x_1^{m-3}x_2^3(\mathcal{U}_m(2,0)-\mathcal{W}_m)+x_1^{m-1}x_4\mathcal{X}_m(1,0)\notag\\
&\hspace{1mm}+x_1^{m-2}x_2x_3(\mathcal{U}_m(0,1)+\mathcal{V}_m(0,2,2))\bigg)\notag.
\end{align*}
Put $\phi'_m := x_5(2m+3)!!$. Then $\widetilde{\phi}'_p = 0$. Lemma \ref{thmDIFF} yields %with $\mathcal{E}_1^* = \frac{x_1^3x_5}{5!!}$ we obtain the explicit solution 
\begin{align}
\mathcal{E}_m &=\phi'_m + x_1^{m+2}(2m+3)!!\sum_{j=1}^{m-1}\frac{j(2j-1)}{3(2j+3)(2j+5)} + \frac{(j+1)(2j+3)}{(2j+5)!!}\mathcal{S}_j(1,-1)\\
&\hspace{3mm}+x_1^{m-1}x_2^2(2m+3)!!\sum_{j=1}^{m-1}\Bigg(\frac{j(2j+1)!!}{(2j+5)!!} + \frac{(j+1)(2j+3)}{(2j+5)!!}(\mathcal{S}_j(2,0)-\mathcal{T}_j)\notag\\
&\hspace{6mm}+\frac{(j+1)(2j+1)}{(2j+5)!!}\mathcal{R}_j +\frac{(2j+4)}{(2j+5)!!}(\mathcal{X}_j(2,0)+\mathcal{V}_j(2,3,1)+\mathcal{U}_j(1,-1))\Bigg)\notag\\
&\hspace{3mm}+x_1^{m-4}x_2^4(2m+3)!!\sum_{j=1}^{m-1}\frac{(2j+4)}{(2j+5)!!}(\mathcal{U}_j(2,0) - \mathcal{W}_j)\notag\\
&\hspace{3mm}+x_1^{m-3}x_2^2x_3(2m+3)!!\sum_{j=1}^{m-1}\frac{(2j+4)}{(2j+5)!!}(\mathcal{U}_j(0,1) + \mathcal{V}_j(0,2,2)) + \frac{2j+3}{(2j+5)!!}(\mathcal{S}_j(2,0)-\mathcal{T}_j)\notag\\
&\hspace{3mm}+x_1^{m}x_3(2m+3)!!\sum_{j=1}^{m-1}\frac{(j+1)(2j+3)}{(2j+5)!!}\mathcal{S}_j(0,1) + \frac{2j}{(2j+3)(2j+5)} + \frac{2j+3}{(2j+5)!!}\mathcal{S}_j(1,-1)\notag\\
&\hspace{3mm}+x_1^{m-2}x_2x_4(2m+3)!!\sum_{j=1}^{m-1}\frac{(2j+4)}{(2j+5)!!}\mathcal{X}_j(1,0) + \frac{2j+2}{(2j+5)!!}\mathcal{R}_j\notag\\
&\hspace{3mm}+x_1^{m-2}x_3^2(2m+3)!!\sum_{j=1}^{m-1}\frac{2j+3}{(2j+5)!!}\mathcal{S}_j(0,1)\notag\\
&\hspace{3mm}+x_1^{m-1}x_5(2m+3)!!\sum_{j=1}^{m-1}\frac{1}{(2j+3)(2j+5)}.\notag
\end{align}
In view of the definitions in Lemma \ref{Ssum} and Lemma \ref{uvxw} we have
\begin{align}\label{Esolve}
\mathcal{E}_m &=\phi_m' + x_1^{m+2}\left(\mathcal{Z}_m(2/3,-1/3,0)+ \widehat{\mathcal{S}}_m(1,-1,1)\right)\\
&\hspace{3mm}+x_1^{m-1}x_2^2\left(\mathcal{Z}_m(0,1,0) + \widehat{\mathcal{S}}_m(2,0,1) - \widehat{\mathcal{T}}_m(1) + \frac{1}{2}\widehat{\mathcal{R}}_m(2) + \mathcal{\widehat{X}}_m(2,0) + \mathcal{\widehat{V}}_m(2,3,1) + \mathcal{\widehat{U}}_m(1,-1)\right)\notag\\
%&\hspace{3mm}+x_1^{m-1}x_2^2(2m+3)!!\left(\widehat{\mathcal{R}}_m + \mathcal{\widehat{X}}_m(2,0) + \mathcal{\widehat{V}}_m(2,3,1) + \mathcal{\widehat{U}}_m(1,-1)\right)\notag\\
&\hspace{3mm}+x_1^{m-4}x_2^4\left(\mathcal{\widehat{U}}_m(2,0) - \mathcal{\widehat{W}}_m\right)+x_1^{m-3}x_2^2x_3(\widehat{\mathcal{U}}_m(0,1) + \widehat{\mathcal{V}}_m(0,2,2) + \widehat{\mathcal{S}}_m(2,0,0) - \widehat{\mathcal{T}}_m(0))\notag\\
&\hspace{3mm}+x_1^{m}x_3(\widehat{\mathcal{S}}_m(0,1,1) + \widehat{\mathcal{Z}}_m(0,2,0) + \widehat{\mathcal{S}}_m(1,-1,0))+x_1^{m-2}x_2x_4(\widehat{\mathcal{X}}_m(0,1) + \widehat{\mathcal{R}}_m(0)) \notag\\
&\hspace{3mm}+x_1^{m-2}x_3^2\widehat{\mathcal{S}}_m(0,1,0)+ x_1^{m-1}x_5\mathcal{Z}(0,0,1)
\notag.
\end{align}
%The next step in the proof is to examine these explicit solutions after $p$ iterations.
Recall that \[\widehat{f}^p(\zeta) \equiv \mathcal{A}_p\zeta^{2p+1} + \mathcal{B}_p\zeta^{2p+2} + \mathcal{C}_p\zeta^{2p+3} + \mathcal{D}_p\zeta^{2p+4} + \mathcal{E}_p\zeta^{2p+5} \mod \langle \zeta^{2p+6} \rangle.\] It follows from (\ref{Asolve}), (\ref{Bsolve}) and Lemma \ref{sumident} that \[\widetilde{\mathcal{A}}_p = \widetilde{\mathcal{B}}_p = 0,\] and from (\ref{Csolve}), Lemma \ref{sumident} and Lemma \ref{sumidentfinitefield} we obtain \[\widetilde{\mathcal{C}}_p = a_1^{p-2}(3/2a_1^3+a_2^2-a_2a_4) = \alpha_1.\] In view of Lemma \ref{dmcoeff} we have \[\widetilde{\mathcal{V}}_p(2,3,1) = -3, \quad \widetilde{\mathcal{U}}_p(1,-1)= 9/2,\quad \widetilde{\mathcal{X}}_p(2,0)=0, \quad \widetilde{\mathcal{U}}_p(2,0)=3, \quad \widetilde{\mathcal{W}}_p =2,\]\[\widetilde{\mathcal{U}}_p(0,1)=-3,\quad\widetilde{\mathcal{V}}_p(0,2,2)=2 \quad \text{ and }\quad\widetilde{\mathcal{X}}_p(1,0)=0.\]  For each $i \geq 1$ we specialize $x_i$ to $a_{i+1}$, and by (\ref{fourtheq}) we obtain 
\begin{align}\label{DsolveP}\widetilde{\mathcal{D}}_{p}(a_2,a_3) &=  (9/2-3)a_2^pa_3 + (3-2)a_2^{p-3}a_3^3\\
&\hspace{3mm}+(-3+2)a_2^{p-2}a_3a_4\notag\\
&= a_2^{p-3}a_3\left(3/2a_2^3 + a_3^2-a_2a_4\right)\notag\\
&= \frac{a_3}{a_2}\widetilde{\mathcal{C}}_p = \frac{a_3}{a_2}\alpha_1 = \beta_1.
\end{align}
Concerning $\mathcal{E}_p$ we have by Lemma \ref{Ssum} and Lemma \ref{uvxw}
\[\widetilde{\mathcal{Z}}_p(2/3,-1/3,0) = \frac{9}{2}, \quad \widetilde{\mathcal{Z}}_p(0,1,0) = -\frac{3}{2},\] \[\widetilde{\widehat{\mathcal{S}}}_p(1,-1,1) = -\frac{27}{4}, \quad \widetilde{\widehat{\mathcal{T}}}_p(1) = -9, \quad \widetilde{\widehat{\mathcal{S}}}_p(2,0,1) = -12,  \]
\[ \widetilde{\widehat{\mathcal{R}}}_p(2) = 15, \quad \widetilde{\mathcal{\widehat{X}}}_p(2,0)= -6, \quad \widetilde{\mathcal{\widehat{V}}}_p(2,3,1) = -\frac{33}{2}, \] \[\widetilde{\mathcal{\widehat{U}}}_p(1,-1) = \frac{21}{2}, \quad \widetilde{\mathcal{\widehat{U}}}_p(2,0) = 2,\quad \widetilde{\mathcal{\widehat{W}}}_p = 2,\] \[\widetilde{\widehat{\mathcal{U}}}_p(0,1)=-\frac{19}{2},\quad \widetilde{\widehat{\mathcal{V}}}_p(0,2,2)=8,\quad \widetilde{\widehat{\mathcal{S}}}_p(0,1,1)=\frac{3}{4},\quad \widetilde{\widehat{\mathcal{Z}}}_p(0,2,0)=-3\] \[\widetilde{\widehat{\mathcal{S}}}_p(1,-1,0)=\frac{21}{4},\quad \widetilde{\widehat{\mathcal{S}}}_p(0,1,0)=-1,\quad \widetilde{\widehat{\mathcal{S}}}_p(2,0,0)=\frac{17}{2},\quad \widetilde{\widehat{\mathcal{T}}}_p(0)=6\]and \[\widetilde{\widehat{\mathcal{X}}}_p(0,1)=3,\quad \widetilde{\widehat{\mathcal{R}}}_p(0)=-3,\quad \widetilde{\widehat{\mathcal{Z}}}_p(0,0,1)=0.\]respectively. Hence, from (\ref{Esolve}) we obtain
\begin{align}\label{EsolveP}
\widetilde{\mathcal{E}}_p(a_2,a_3) &=\left(\frac{9}{2}-\frac{27}{4}\right)a_2^{p+2} + \left(-\frac{3}{2}-12+9+\frac{15}{2}-6-\frac{33}{2}+\frac{21}{2}\right)a_2^{p-1}a_3^2 \notag\\
&\hspace{5mm}+ (2-2)a_2^{p-4}a_3^4 + \left(-\frac{19}{2} + 8 + \frac{17}{2}-6\right)a_2^{p-3}a_3^2a_4 \notag\\
&\hspace{5mm}+ \left(\frac{3}{4}-3+\frac{21}{4}\right)a_2^{p}a_4 -a_2^{p-2}a_4^2+ (3-3)a_2^{p-2}a_3a_5\notag\\
&=-9/4a_2^{p+2}  -3/2a_2^{p-1}a_3^2+a_2^{p-3}a_3^2a_4 +3a_2^pa_4-a_2^{p-2}a_4^2\notag\\
&= -\left(\frac{3a_2}{2}-\frac{a_4}{a_2}\right)\widetilde{\mathcal{C}}_p = -\left(\frac{3a_2}{2}-\frac{a_4}{a_2}\right)\alpha_1 = \gamma_1.
\end{align}
This completes the proof of the Main Lemma.
%By Lemma \ref{lemmaevenfact} we can find the values of $\mathcal{V}_p, \mathcal{U}_p, \mathcal{X}_p$ and $\mathcal{W}_p$. Hence, we have
%\begin{align}\widetilde{D}_{p} &=  x_1^px_2(\mathcal{V}_p(2,3,1)+\mathcal{U}_p(1,-1)+ \mathcal{X}_p(2,0)) 
%\notag\\&\hspace{5mm}+ x_1^{p-3}x_2^3(\mathcal{U}_p(2,0)-\mathcal{W}_p)\notag\\
%&=3/2x_1^px_2+x_1^{p-3}x_2^3\notag\\
%&=\frac{x_2}{x_1}\widetilde{C}_p\notag
%\end{align}
\end{proof}

\begin{proof}[Proof of Theorem \ref{pncoeff}]
As for the case of the proof of the Main Lemma we divide this proof into three cases, thus treating the special cases $p=3$ and $p=5$ separately. 

By the Main Lemma the theorem is valid for $n=1$. Assume that the theorem is valid for some integer $n \geq 1$. Let $h(\zeta) = f^{p^n}(\zeta)$. Then, \[h(\zeta)-\zeta = \alpha_n\zeta^{2d+1} + \beta_n\zeta^{2d+2} + \gamma_n\zeta^{2d+3} + \langle \zeta^{2d+4} \rangle.\] For each integer $m \geq 1$ define the power series $\Delta_m(\zeta)$ inductively by $\Delta_1(\zeta):=h(\zeta)-\zeta$, and for $m\geq 2$ by \[\Delta_m(\zeta) := \Delta_{m-1}(h(\zeta))-\Delta_{m-1}(\zeta).\] Note that $\Delta_p(\zeta) = h^p(\zeta)-\zeta = f^{p^{n+1}}(\zeta)-\zeta$. 

As in the proof of the Main Lemma we define $F_\ell := \mathbb{Q}_p[x_1,x_2,\dots,x_\ell]$, and $F_\infty := \mathbb{Q}_p[x_1,x_2,\dots]$. We recall that $d = \frac{p^n-1}{p-1}$ and put \begin{equation}\label{alphabetadef}\widehat{\varphi} = \frac{3}{2}x_1^3+x_2^2-x_1x_3,\quad \widehat{\alpha}_n=\widehat{\alpha}_n(x) :=x_1^{p^n-2d}\widehat{\varphi}^d,\end{equation} and \begin{equation}\label{gammadef}\quad \widehat{\beta}_n= \widehat{\beta}_n(x) := \frac{x_2}{x_1}\widehat{\alpha}_n(x), \quad \widehat{\gamma}_n =\widehat{\gamma}_n(x) := -\left(\frac{3x_1}{2}-{x_3}{x_1}\right)\widehat{\alpha}_n.\end{equation} For future reference we note that \begin{equation}\label{betaalphagamma} \widehat{\beta}_n^2 - \widehat{\alpha}_n\widehat{\gamma}_n = x_1^{-2}\widehat{\alpha}_n^2\widehat{\varphi}.\end{equation} Moreover we consider the power series $\widehat{h} \in F_\infty[[\zeta]]$ defined as \[\widehat{h}(\zeta)-\zeta := \widehat{\alpha}_n\zeta^{2d+1} + \widehat{\beta}_n\zeta^{2d+2} + \widehat{\gamma}_n\zeta^{2d+3} + \langle \zeta^{2d+4} \rangle.\] 
%i.e. such that for each integer $i\geq 1$ the evaluation mapping $x_i \mapsto a_i$ maps $\widehat{g}$ to $g$. 
For $m=1$ we define the relation $\widehat{\Delta}_1(\zeta) := \widehat{h}(\zeta)-\zeta$ and for each integer $m\geq 2$ \[\widehat{\Delta}_m(\zeta) := \widehat{\Delta}_{m-1}(\widehat{h}(\zeta))-\widehat{\Delta}_{m-1}(\zeta).\]  
We will prove that the first significant terms of $\widetilde{\Delta}_p(\zeta)$ are $\alpha_{n+1}, \beta_{n+1}$ and $\gamma_{n+1}$.

As for the proof of the Main Lemma in the cases of $p=3$ and $p=5$ we redefine the ground field of $F_\ell$ to be $\mathbb{F}_p$ for $p=3$ and $p=5$ respectively, and utilize, without comment, that the field characteristic is 3 and 5.

\partn{Case 1, $p=3$} Let $ \widehat{\delta}_n$ and $ \widehat{\varepsilon}_n$ denote the coefficients of the terms of order $2d+4$ and $2d+5$ in $\widehat{h}$ respectively. By definition $d\equiv 1 \pmod{p}$. %As for the Main Lemma we explicitly compute $\widehat{\Delta}_3$. 
We obtain
\begin{align*}
\widehat{\Delta}_2(\zeta) &\equiv \widehat{\Delta}_1(\widehat{h}(\zeta)) - \widehat{\Delta}_1(\zeta)\\
&\equiv \widehat{\alpha}_n[(\zeta+\widehat{\Delta}_1)^{2d+1}-\zeta^{2d+1}]+  \widehat{\beta}_n[(\zeta+\widehat{\Delta}_1)^{2d+2}-\zeta^{2d+2}]\\
&\hspace{5mm}+ \widehat{\gamma}_n[(\zeta+\widehat{\Delta}_1)^{2d+3}-\zeta^{2d+3}]+ \widehat{\delta}_n[(\zeta+\widehat{\Delta}_1)^{2d+4}-\zeta^{2d+4}]\\
&\hspace{5mm}+ \widehat{\varepsilon}_n[(\zeta+\widehat{\Delta}_1)^{2d+5}-\zeta^{2d+5}]\\
&\equiv3 \widehat{\alpha}_n(\dots) + 4 \widehat{\alpha}_n \widehat{\beta}_n\zeta^{4d+2} + 4 \widehat{\beta}_n^2\zeta^{4d+3} + 4 \widehat{\beta}_n \widehat{\gamma}_n\zeta^{4d+4} + 4 \widehat{\beta}_n \widehat{\delta}_n\zeta^{4d+5}\\
&\hspace{5mm}+5 \widehat{\alpha}_n \widehat{\gamma}_n\zeta^{4d+3} + 5 \widehat{\beta}_n \widehat{\gamma}_n\zeta^{4d+4} + 5 \widehat{\gamma}_n^2\zeta^{4d+5} + 6 \widehat{\delta}_n(\dots) + 7 \widehat{\alpha}_n \widehat{\varepsilon}_n\zeta^{4d+5}\\
&\equiv \widehat{\alpha}_n \widehat{\beta}_n\zeta^{4d+2} + ( \widehat{\beta}_n^2 + 2 \widehat{\alpha}_n \widehat{\gamma}_n)\zeta^{4d+3} + ( \widehat{\alpha}_n \widehat{\varepsilon}_n + 2 \widehat{\gamma}_n^2 + \widehat{\beta}_n \widehat{\delta}_n)\zeta^{4d+5}\mod \langle \zeta^{4d+6} \rangle.
\end{align*}
Also,
\begin{align*}
\widehat{\Delta}_3(\zeta) &\equiv \widehat{\Delta}_2(\widehat{h}(\zeta)) - \widehat{\Delta}_2(\zeta)\\
&\equiv \widehat{\alpha}_n \widehat{\beta}_n[(\zeta+\widehat{\Delta}_1)^{4d+2}-\zeta^{4d+2}]+ ( \widehat{\beta}_n^2+2 \widehat{\alpha}_n \widehat{\gamma}_n)[(\zeta+\widehat{\Delta}_1)^{4d+3}-\zeta^{4d+3}]\\
&\hspace{5mm}+( \widehat{\alpha}_n \widehat{\varepsilon}_n + 2 \widehat{\gamma}_n^2 +  \widehat{\beta}_n \widehat{\delta}_n)[(\zeta+\widehat{\Delta}_1)^{4d+5}-\zeta^{4d+5}]\\
&6 \widehat{\alpha}_n \widehat{\beta}_n(\dots) + 7 \widehat{\alpha}_n( \widehat{\beta}_n^2+2 \widehat{\alpha}_n \widehat{\gamma}_n)\zeta^{6d+3} + 7 \widehat{\beta}_n( \widehat{\beta}_n^2+2 \widehat{\alpha}_n \widehat{\gamma}_n)\zeta^{6d+4}\\
&\hspace{5mm}+7 \widehat{\gamma}_n( \widehat{\beta}_n^2+2 \widehat{\alpha}_n \widehat{\gamma}_n)\zeta^{6d+5} + 9( \widehat{\alpha}_n \widehat{\varepsilon}_n + 2 \widehat{\gamma}_n^2 +  \widehat{\beta}_n \widehat{\delta}_n)(\dots)\\
&\equiv \widehat{\alpha}_n( \widehat{\beta}_n^2- \widehat{\alpha}_n \widehat{\gamma}_n)\zeta^{6d+3} +  \widehat{\beta}_n( \widehat{\beta}_n^2- \widehat{\alpha}_n \widehat{\gamma}_n)\zeta^{6d+4} +  \widehat{\gamma}_n( \widehat{\beta}_n^2- \widehat{\alpha}_n \widehat{\gamma}_n)\zeta^{6d+5}\mod \langle \zeta^{6d+6} \rangle.
\end{align*}
Recall that $d=\frac{p^n-1}{p-1}$  by which we obtain $dp+1 = \frac{p^{n+1}-1}{p-1}$ and thus we see that 
\[\widehat{\Delta}_3(\zeta) \equiv  \widehat{\alpha}_n( \widehat{\beta}_n^2- \widehat{\alpha}_n \widehat{\gamma}_n)\zeta^{2(3d+1)+1} +  \widehat{\beta}_n( \widehat{\beta}_n^2- \widehat{\alpha}_n \widehat{\gamma}_n)\zeta^{2(3d+1)+2} +  \widehat{\gamma}_n( \widehat{\beta}_n^2- \widehat{\alpha}_n \widehat{\gamma}_n)\zeta^{2(3d+1)+3} \mod \langle \zeta^{2(3d+1)+4} \rangle.\]
In view of (\ref{betaalphagamma}) we have \[ \widehat{\alpha}_n( \widehat{\beta}_n^2- \widehat{\alpha}_n \widehat{\gamma}_n) =  \widehat{\alpha}_n^3x_1^{-2}\widehat{\varphi} = x_1^{-2}\left(x_1^{3^n-2d}\widehat{\varphi}^d\right)^3\widehat{\varphi} = x_1^{3^{n+1}-2(3d+1)}\widehat{\varphi}^{3d+1} = \widehat{\alpha}_{n+1},\]  \[ \widehat{\beta}_n( \widehat{\beta}_n^2- \widehat{\alpha}_n \widehat{\gamma}_n) = \frac{x_2}{x_1}\widehat{\alpha}_{n+1} = \widehat{\beta}_{n+1},\] and \[ \widehat{\gamma}_n( \widehat{\beta}_n^2- \widehat{\alpha}_n \widehat{\gamma}_n) = \frac{x_3}{x_1}\widehat{\alpha}_{n+1}= \widehat{\gamma}_{n+1}.\] The proof is finished by specializing for each $i\in \{1,2,3\}$  $x_i$ to $a_{i+1}$ from which we obtain \[\widetilde{\Delta}_3(\zeta) = f^{3^{n+1}}(\zeta) -\zeta =  \alpha_{n+1}\zeta^{2(3d+1)+1} + \beta_{n+1}\zeta^{2(3d+1)+2} + \gamma_{n+1}\zeta^{2(3d+1)+3} \mod \langle \zeta^{2(3d+1)+4} \rangle.\]
%\[ \widehat{\alpha}_n \widehat{\beta}_n^2 =  \widehat{\alpha}_n^3\frac{a_3^2}{a_2^2} = \left(a_2^{3^n-2d}\varphi^d\right)^3\frac{a_3^2}{a_2^2} = a_2^{3^{n+1}-2(3d+1)}\varphi^{3d+1} = \alpha_{n+1},\] and \[ \widehat{\beta}_n^3 = \frac{a_3}{a_2}\alpha_{n+1} = \beta_{n+1},\] and finally \[
This completes the induction and thus the proof of the Theorem for the case $p=3$.

\partn{Case 2, $p=5$} As for the case of $p=3$ we let $ \widehat{\delta}_n$ and $ \widehat{\varepsilon}_n$ denote the coefficients of the terms of order $2d+4$ and $2d+5$ in $\widehat{h}$ respectively. Recall that $d\equiv 1 \pmod{p}$. %By explicitly computing $\widehat{\Delta}_2$
 We obtain
\begin{align*}
\widehat{\Delta}_2(\zeta) &\equiv \widehat{\Delta}_1(\widehat{h}(\zeta)) - \widehat{\Delta}_1(\zeta)\\
&\equiv \widehat{\alpha}_n[(\zeta+\widehat{\Delta}_1)^{2d+1}-\zeta^{2d+1}]+  \widehat{\beta}_n[(\zeta+\widehat{\Delta}_1)^{2d+2}-\zeta^{2d+2}]\\
&\hspace{5mm}+ \widehat{\gamma}_n[(\zeta+\widehat{\Delta}_1)^{2d+3}-\zeta^{2d+3}]+ \widehat{\delta}_n[(\zeta+\widehat{\Delta}_1)^{2d+4}-\zeta^{2d+4}]\\
&\hspace{5mm}+ \widehat{\varepsilon}_n[(\zeta+\widehat{\Delta}_1)^{2d+5}-\zeta^{2d+5}]\\
&\equiv3 \widehat{\alpha}_n^2\zeta^{4d+1} + 3 \widehat{\alpha}_n \widehat{\beta}_n\zeta^{4d+2} + 3 \widehat{\alpha}_n \widehat{\gamma}_n\zeta^{4d+3} + 3 \widehat{\alpha}_n \widehat{\delta}_n\zeta^{4d+4}+3 \widehat{\alpha}_n \widehat{\varepsilon}_n\zeta^{4d+5} \\
&\hspace{5mm}+ 4 \widehat{\alpha}_n \widehat{\beta}_n\zeta^{4d+2} + 4 \widehat{\beta}_n^2\zeta^{4d+3} + 4 \widehat{\beta}_n \widehat{\gamma}_n\zeta^{4d+4} + 4 \widehat{\beta}_n \widehat{\delta}_n\zeta^{4d+5}\\
&\hspace{5mm}+5 \widehat{\gamma}_n(\dots)+ 6 \widehat{\alpha}_n \widehat{\delta}_n\zeta^{4d+4} + 6 \widehat{\beta}_n \widehat{\delta}_n\zeta^{4d+5} + 7 \widehat{\alpha}_n \widehat{\varepsilon}_n\zeta^{4d+5}\\
&\equiv3 \widehat{\alpha}_n^2\zeta^{4d+1} + 2 \widehat{\alpha}_n \widehat{\beta}_n\zeta^{4d+2} + (4 \widehat{\beta}_n^2 + 3 \widehat{\alpha}_n \widehat{\gamma}_n)\zeta^{4d+3} \\
&\hspace{5mm}+4( \widehat{\alpha}_n \widehat{\delta}_n+ \widehat{\beta}_n \widehat{\gamma}_n)\zeta^{4d+4}  \mod \langle \zeta^{4d+6} \rangle.
\end{align*}
Continuing for $\widehat{\Delta}_3$ we have
\begin{align*}
\widehat{\Delta}_3(\zeta) &\equiv \widehat{\Delta}_2(\widehat{h}(\zeta)) - \widehat{\Delta}_2(\zeta)\\
&\equiv3 \widehat{\alpha}_n^2[(\zeta+\widehat{\Delta}_1)^{4d+1}-\zeta^{4d+1}]+ 2 \widehat{\alpha}_n \widehat{\beta}_n[(\zeta+\widehat{\Delta}_1)^{4d+2}-\zeta^{4d+2}]\\
&\hspace{5mm}+(4 \widehat{\beta}_n^2 + 3 \widehat{\alpha}_n \widehat{\gamma}_n)[(\zeta+\widehat{\Delta}_1)^{4d+3}-\zeta^{4d+3}]+( \widehat{\alpha}_n \widehat{\delta}_n+4 \widehat{\beta}_n \widehat{\gamma}_n)[(\zeta+\widehat{\Delta}_1)^{4d+4}-\zeta^{4d+4}]\\
&\equiv15 \widehat{\alpha}_n^2(\dots) + 12 \widehat{\alpha}_n^2 \widehat{\beta}_n\zeta^{6d+2} + 12 \widehat{\alpha}_n \widehat{\beta}_n^2\zeta^{6d+3} + 12 \widehat{\alpha}_n \widehat{\beta}_n \widehat{\gamma}_n\zeta^{6d+4} + 12 \widehat{\alpha}_n \widehat{\beta}_n \widehat{\delta}_n\zeta^{6d+5}\\
&\hspace{5mm}+7 \widehat{\alpha}_n(4 \widehat{\beta}_n^2 + 3 \widehat{\alpha}_n \widehat{\gamma}_n)\zeta^{6d+3} + 7 \widehat{\beta}_n(4 \widehat{\beta}_n^2 + 3 \widehat{\alpha}_n \widehat{\gamma}_n)\zeta^{6d+4} + 7 \widehat{\gamma}_n(4 \widehat{\beta}_n^2 + 3 \widehat{\alpha}_n \widehat{\gamma}_n)\zeta^{6d+5}\\
&\hspace{5mm}+32 \widehat{\alpha}_n( \widehat{\alpha}_n \widehat{\delta}_n+ \widehat{\beta}_n \widehat{\gamma}_n)\zeta^{6d+4} + 32 \widehat{\beta}_n( \widehat{\alpha}_n \widehat{\delta}_n+ \widehat{\beta}_n \widehat{\gamma}_n)\zeta^{6d+5} \\
&\equiv2 \widehat{\alpha}_n^2 \widehat{\beta}_n\zeta^{6d+2} +  \widehat{\alpha}_n^2 \widehat{\gamma}_n\zeta^{6d+3} + (3 \widehat{\beta}_n^3+2 \widehat{\alpha}_n^2 \widehat{\delta}_n)\zeta^{6d+4}\\
&\hspace{5mm} + ( \widehat{\alpha}_n \widehat{\gamma}_n^2+4 \widehat{\alpha}_n \widehat{\beta}_n \widehat{\delta}_n)\zeta^{6d+5} \mod \langle \zeta^{6d+6} \rangle.
 \end{align*}
For $\widehat{\Delta}_4$ we have
\begin{align*}
\widehat{\Delta}_4(\zeta) &\equiv \widehat{\Delta}_3(\widehat{h}(\zeta)) - \widehat{\Delta}_3(\zeta)\\
&\equiv2 \widehat{\alpha}_n^2 \widehat{\beta}_n[(\zeta+\widehat{\Delta}_1)^{6d+2}-\zeta^{6d+2}]+ \widehat{\alpha}_n^2 \widehat{\gamma}_n[(\zeta+\widehat{\Delta}_1)^{6d+3}-\zeta^{6d+3}]\\
&\hspace{5mm}+(3 \widehat{\beta}_n^3+2 \widehat{\alpha}_n^2 \widehat{\delta}_n)[(\zeta+\widehat{\Delta}_1)^{6d+4}-\zeta^{6d+4}]+( \widehat{\alpha}_n \widehat{\gamma}_n^2+4 \widehat{\alpha}_n \widehat{\beta}_n \widehat{\delta}_n)[(\zeta+\widehat{\Delta}_1)^{6d+5}-\zeta^{6d+5}]\\
&\equiv16 \widehat{\alpha}_n^3 \widehat{\beta}_n\zeta^{8d+2} + 16 \widehat{\alpha}_n^2 \widehat{\beta}_n^2\zeta^{8d+3} + 16 \widehat{\alpha}_n^2 \widehat{\beta}_n \widehat{\gamma}_n\zeta^{8d+4}+16 \widehat{\alpha}_n^2 \widehat{\beta}_n \widehat{\delta}_n\zeta^{8d+5}\\
&\hspace{5mm}+9 \widehat{\alpha}_n^3 \widehat{\gamma}_n\zeta^{8d+3} + 9 \widehat{\alpha}_n^2 \widehat{\beta}_n \widehat{\gamma}_n\zeta^{8d+4} + 9 \widehat{\alpha}_n^2 \widehat{\gamma}_n^2\zeta^{8d+5}\\
&\hspace{5mm} + 10(3 \widehat{\beta}_n^3+2 \widehat{\alpha}_n^2 \widehat{\delta}_n)(\dots) + 11 \widehat{\alpha}_n( \widehat{\alpha}_n \widehat{\gamma}_n^2+4 \widehat{\alpha}_n \widehat{\beta}_n \widehat{\delta}_n)\zeta^{8d+5}\\
&\equiv \widehat{\alpha}_n^3 \widehat{\beta}_n\zeta^{8d+2} +  \widehat{\alpha}_n^2( \widehat{\beta}_n^2+4 \widehat{\alpha}_n \widehat{\gamma}_n)\zeta^{8d+3} \mod \langle \zeta^{8d+6} \rangle.
 \end{align*}
Finally for $\widehat{\Delta}_5$ we have
\begin{align*}
\widehat{\Delta}_5(\zeta) &\equiv \widehat{\Delta}_4(\widehat{h}(\zeta)) - \widehat{\Delta}_4(\zeta)\\
&\equiv \widehat{\alpha}_n^3 \widehat{\beta}_n[(\zeta+\widehat{\Delta}_1)^{8d+2}-\zeta^{8d+2}]+ \widehat{\alpha}_n^2( \widehat{\beta}_n^2+4 \widehat{\alpha}_n \widehat{\gamma}_n)[(\zeta+\widehat{\Delta}_1)^{8d+3}-\zeta^{8d+3}]\\
%&\hspace{5mm}+4 \widehat{\alpha}_n^2 \widehat{\beta}_n \widehat{\gamma}_n[(\zeta+\widehat{\Delta}_1)^{8d+4}-\zeta^{8d+4}]+ \widehat{\alpha}_n^2 \widehat{\beta}_n \widehat{\delta}_n[(\zeta+\widehat{\Delta}_1)^{8d+5}-\zeta^{8d+5}]\\
&\equiv10 \widehat{\alpha}_n^3 \widehat{\beta}_n(\dots) + 11 \widehat{\alpha}_n^3( \widehat{\beta}_n^2+4 \widehat{\alpha}_n \widehat{\gamma}_n)\zeta^{10d+3} + 11 \widehat{\alpha}_n^2 \widehat{\beta}_n( \widehat{\beta}_n^2+4 \widehat{\alpha}_n \widehat{\gamma}_n)\zeta^{10d+4}\\
&\hspace{5mm} + 11 \widehat{\alpha}_n^2 \widehat{\gamma}_n( \widehat{\beta}_n^2+4 \widehat{\alpha}_n \widehat{\gamma}_n)\zeta^{10d+5}\\%+48 \widehat{\alpha}_n^3 \widehat{\beta}_n \widehat{\gamma}_n\zeta^{10d+4} + 48 \widehat{\alpha}_n^2 \widehat{\beta}_n^2 \widehat{\gamma}_n\zeta^{10d+5}\\
%&\hspace{5mm}+13 \widehat{\alpha}_n^2 \widehat{\beta}_n \widehat{\delta}_n\zeta^{10d+5}\\
&\equiv \widehat{\alpha}_n^3( \widehat{\beta}_n^2- \widehat{\alpha}_n \widehat{\gamma}_n)\zeta^{10d+3} +  \widehat{\alpha}_n^2 \widehat{\beta}_n( \widehat{\beta}_n^2- \widehat{\alpha}_n \widehat{\gamma}_n)\zeta^{10d+4}\\
&\hspace{5mm} +  \widehat{\alpha}_n^2 \widehat{\gamma}_n( \widehat{\beta}_n^2- \widehat{\alpha}_n \widehat{\gamma}_n)\zeta^{10d+5}\mod \langle \zeta^{10d+6} \rangle.
 \end{align*}
 In view of (\ref{betaalphagamma}) we have \[ \widehat{\alpha}_n^3( \widehat{\beta}_n^2- \widehat{\alpha}_n \widehat{\gamma}_n) =  \widehat{\alpha}_n^5x_1^{-2}\widehat{\varphi} = x_1^{-2}\left(x_1^{5^n-2d}\widehat{\varphi}^d\right)^5\widehat{\varphi} = x_1^{5^{n+1}-2(5d+1)}\widehat{\varphi}^{5d+1} = \widehat{\alpha}_{n+1},\]
\[ \widehat{\alpha}_n^2 \widehat{\beta}_n( \widehat{\beta}_n^2- \widehat{\alpha}_n \widehat{\gamma}_n) = \frac{x_2}{x_1}\widehat{\alpha}_{n+1} =\widehat{\beta}_{n+1},\] and finally \[ \widehat{\alpha}_n^2 \widehat{\gamma}_n( \widehat{\beta}_n^2- \widehat{\alpha}_n \widehat{\gamma}_n) = -\left(\frac{3x_1}{2}-\frac{x_3}{x_1}\right)\widehat{\alpha}_{n+1} = \widehat{\gamma}_{n+1}.\] The proof is completed by specializing for each $i \in [1,3]$ $x_i$ to $a_{i+1}$, which yields \[\widetilde{\Delta}_5(\zeta) = f^{5^{n+1}}(\zeta) -\zeta =  \alpha_{n+1}\zeta^{2(5d+1)+1} + \beta_{n+1}\zeta^{2(5d+1)+2} + \gamma_{n+1}\zeta^{2(5d+1)+3} \mod \langle \zeta^{2(5d+1)+4} \rangle.\]
%This completes the proof of the Theorem for the case $p=5$.

\partn{Case 3, $p\geq7$}
We recall that \[\widehat{\varphi} = \frac{3}{2}x_1^3+x_2^2-x_1x_3,\quad \widehat{\alpha}_n=\widehat{\alpha}_n(x) :=x_1^{p^n-2d}\widehat{\varphi}^d,\] and \[\quad \widehat{\beta}_n= \widehat{\beta}_n(x) := \frac{x_2}{x_1}\widehat{\alpha}_n(x), \quad \widehat{\gamma}_n =\widehat{\gamma}_n(x) := -\left(\frac{3x_1}{2}-{x_3}{x_1}\right)\widehat{\alpha}_n.\]
Also let $ \widehat{\delta}_n,  \widehat{\varepsilon}_n$ be the coefficients of the terms of degree $2d+4$ and $2d+5$ in $\widehat{h}$ respectively, and $x = ( x_1,x_2,\dots)$. We will prove that for a given $m \geq 1$ we have 
\begin{equation}\label{deltamlast}\widehat{\Delta}_m(\zeta) \equiv \widehat{\mathcal{A}}_m\zeta^{2dm + 1} + \widehat{\mathcal{B}}_m\zeta^{2dm + 2} + \widehat{\mathcal{C}}_m\zeta^{2dm + 3} + \widehat{\mathcal{D}}_m\zeta^{2dm+4} + \widehat{\mathcal{E}}_m\zeta^{2dm+5} \mod \langle \zeta^{2dm+6}\rangle.\end{equation} where $\widehat{\mathcal{A}}=\widehat{\mathcal{A}}(x), \widehat{\mathcal{B}}=\widehat{\mathcal{B}}(x), \widehat{\mathcal{C}}=\widehat{\mathcal{C}}(x), \widehat{\mathcal{D}}=\widehat{\mathcal{D}}(x)$ and $\widehat{\mathcal{E}}=\widehat{\mathcal{E}}(x)$ are solutions of %the following system of difference equations
\begin{equation}\label{matrix}
\begin{bmatrix}
   \widehat{\alpha}_n(2m + 1) & 0 & 0 & 0 & 0\\
   \widehat{\beta}_n(2m + 1) & \  \widehat{\alpha}_n(2m+2) & 0 &0 & 0\\
   \widehat{\gamma}_n(2m + 1) &  \widehat{\beta}_n(2m+2) &  \widehat{\alpha}_n(2m + 3) &0 & 0\\
     \widehat{\delta}_n(2m + 1) &  \widehat{\gamma}_n(2m+2) &  \widehat{\beta}_n(2m + 3) &  \widehat{\alpha}_n(2m+4) & 0\\
         \widehat{\varepsilon}_n(2m + 1) &  \widehat{\delta}_n(2m+2) &  \widehat{\gamma}_n(2m + 3) &  \widehat{\beta}_n(2m+4) &  \widehat{\alpha}_n(2m+5)\\
 \end{bmatrix}
 \begin{bmatrix}
  \widehat{\mathcal{A}}_{m} \\
 \widehat{\mathcal{B}}_{m} \\
 \widehat{\mathcal{C}}_{m} \\
  \widehat{\mathcal{D}}_{m} \\
   \widehat{\mathcal{E}}_{m} 
 \end{bmatrix}
 =
 \begin{bmatrix}
 \widehat{\mathcal{A}}_{m+1} \\
 \widehat{\mathcal{B}}_{m+1} \\
 \widehat{\mathcal{C}}_{m+1} \\
  \widehat{\mathcal{D}}_{m+1} \\
   \widehat{\mathcal{E}}_{m+1} \
 \end{bmatrix},
 \end{equation}
with initial conditions $(\widehat{\mathcal{A}}_{1} , \widehat{\mathcal{B}}_{1} , \widehat{\mathcal{C}}_{1}, \widehat{\mathcal{D}}_1, \widehat{\mathcal{E}}_1) = ( \widehat{\alpha}_n,  \widehat{\beta}_n,  \widehat{\gamma}_n,  \widehat{\delta}_n,  \widehat{\varepsilon}_n)$. We will proceed by induction in $m$. For $m=1$ (\ref{deltamlast}) holds by definition. Let $m\geq1$ be such that (\ref{deltamlast}) holds. Again recall that $d\equiv1 \pmod{p}$. Using that for $p\geq7$ we have $2d \geq 2(1+7) = 16$ for all $n\geq1$ we obtain
\begin{align*}
\widehat{\Delta}_{m+1}(\zeta) &\equiv \widehat{\Delta}_{m}(\widehat{h}(\zeta)) - \widehat{\Delta}_m(\zeta)\\
&\equiv\widehat{\mathcal{A}}_m[(\zeta + \widehat{\Delta}_1)^{2dm+1}-\zeta^{2dm+1}] + \widehat{\mathcal{B}}_m[(\zeta + \widehat{\Delta}_1)^{2dm+2}-\zeta^{2dm+2}]\\
&\hspace{5mm}+\widehat{\mathcal{C}}_m[(\zeta + \widehat{\Delta}_1)^{2dm+3}-\zeta^{2dm+3}] + \widehat{\mathcal{D}}_m[(\zeta + \widehat{\Delta}_1)^{2dm+4}-\zeta^{2dm+4}]\\
&\hspace{5mm}+\widehat{\mathcal{E}}_m[(\zeta + \widehat{\Delta}_1)^{2dm+5}-\zeta^{2dm+5}]\\
&\equiv\widehat{\mathcal{A}}_m\bigg[ \widehat{\alpha}_n(2m+1)\zeta^{2dm+3} +  \widehat{\beta}_n(2m+1)\zeta^{2dm+4} +  \widehat{\gamma}_n(2m+1)\zeta^{2dm+5} \\
&\hspace{5mm}+  \widehat{\delta}_n(2m+1)\zeta^{2dm+6} +  \widehat{\varepsilon}_n(2m+1)\zeta^{2dm+7}\bigg]\\
&\hspace{5mm}+\widehat{\mathcal{B}}_m\bigg[ \widehat{\alpha}_n(2m+2)\zeta^{2dm+4} +  \widehat{\beta}_n(2m+2)\zeta^{2dm+5} \\
&\hspace{5mm} +  \widehat{\gamma}_n(2m+2)\zeta^{2dm+6}+ \widehat{\delta}_n(2m+2)\zeta^{2dm+7}\bigg]\\
&\hspace{5mm}+\widehat{\mathcal{C}}_m\bigg[ \widehat{\alpha}_n(2m+3)\zeta^{2dm+5} +  \widehat{\beta}_n(2m+3)\zeta^{2dm+6} +  \widehat{\gamma}_n(2m+3)\zeta^{2dm+7}\bigg]\\
&\hspace{5mm}+\widehat{\mathcal{D}}_m\bigg[ \widehat{\alpha}_n(2m+4)\zeta^{2dm+6} +  \widehat{\beta}_n(2m+4)\zeta^{2dm+7}\bigg]+\widehat{\mathcal{E}}_m\bigg[ \widehat{\alpha}_n(2m+5)\zeta^{2dm+7}\bigg]\\
&\equiv \widehat{\mathcal{A}}_{m+1}\zeta^{2dm+3}+\widehat{\mathcal{B}}_{m+1}\zeta^{2dm+4}+\widehat{\mathcal{C}}_{m+1}\zeta^{2dm+5}\\
&\hspace{5mm}+\widehat{\mathcal{D}}_{m+1}\zeta^{2dm+6}+\widehat{\mathcal{E}}_{m+1}\zeta^{2dm+7} \mod \langle \zeta^{2dm+8} \rangle.
\end{align*}
This completes the proof of the induction step and proves (\ref{deltamlast}). 

% and we note that the difference equations only depend on the coefficients $x_1,x_2,x_3,x_4,x_5$ and not those for higher order terms, so $\widehat{\mathcal{A}}_m,\widehat{\mathcal{B}}_m,\widehat{\mathcal{C}}_m,\widehat{\mathcal{D}}_m,\widehat{\mathcal{E}}_m \in F_5$

The first three equations in (\ref{matrix}) were solved in \cite[page 267-268ff]{Fransson2017}. The solutions are 
\begin{equation}\label{AAA}\widehat{\mathcal{A}}_m =  \widehat{\alpha}_n^m(2m-1)!!,
\end{equation} %and 
\begin{equation}\label{BBB}\widehat{\mathcal{B}}_m =  \widehat{\alpha}_n^{m-1} \widehat{\beta}_n\mathcal{R}_m,
\end{equation} 
%The third equation is a special case also solved in \cite[page 268ff]{Fransson2017}. 
%The solution is 
%and
\begin{equation}\label{chatten}\widehat{\mathcal{C}}_m =  \widehat{\alpha}_n^{m-2} \widehat{\beta}_n^2\left(\mathcal{S}_m(2,0)-\mathcal{T}_m\right) +  \widehat{\alpha}_n^{m-1} \widehat{\gamma}_n\mathcal{S}_m(0,1).
\end{equation}
By substituting in $\widehat{\mathcal{D}}_m$ in(\ref{matrix}) for the above equations we obtain \begin{align*}\widehat{\mathcal{D}}_{m+1} &=  \widehat{\alpha}_n(2m+4)\widehat{\mathcal{D}}_m +  \widehat{\alpha}_n^m \widehat{\delta}_n(2m+1)!! +  \widehat{\alpha}_n^{m-1} \widehat{\beta}_n \widehat{\gamma}_n(2m+2)\mathcal{R}_m \\ &\hspace{5mm}+  \widehat{\alpha}_n^{m-2} \widehat{\beta}_n(2m+3)( \widehat{\beta}_n^2\left(\mathcal{S}_m(2,0)-\mathcal{T}_m\right) +  \widehat{\alpha}_n \widehat{\gamma}_n\mathcal{S}_m(0,1)),\end{align*}
which, by Lemma \ref{dmcoeff} and Lemma \ref{thmDIFF}  has the explicit solution \begin{equation}\label{dhatten}\widehat{\mathcal{D}}_m =  \widehat{\alpha}_n^{m-3}( \widehat{\alpha}_n^2 \widehat{\delta}_n\mathcal{X}_m(0,1) +  \widehat{\alpha}_n \widehat{\beta}_n \widehat{\gamma}_n(\mathcal{V}_m(0,2,2) + \mathcal{U}_m(0,1)) +  \widehat{\beta}_n^3(\mathcal{U}_m(2,0)-\mathcal{W}_m)).\end{equation}
Insertion of all the above equations into the fifth equation in (\ref{matrix}) yields
\begin{align*} \widehat{\mathcal{E}}_{m+1} &=  \widehat{\alpha}_n(2m+5)\widehat{\mathcal{E}}_m +  \widehat{\alpha}_n^{m-3}\bigg( \widehat{\alpha}_n^3 \widehat{\varepsilon}_n(2m+1)!! +  \widehat{\alpha}_n^2 \widehat{\beta}_n \widehat{\delta}_n(2m+2)\mathcal{R}_m\\ 
&\hspace{5mm}+ \widehat{\alpha}_n \widehat{\beta}_n^2 \widehat{\gamma}_n(2m+3)(\mathcal{S}_m(2,0)-\mathcal{T}_m) +  \widehat{\alpha}_n^2 \widehat{\gamma}_n^2(2m+3)\mathcal{S}_m(0,1)\\
&\hspace{5mm}+ \widehat{\beta}_n(2m+4)\left( \widehat{\alpha}_n^2 \widehat{\delta}_n\mathcal{X}_m(0,1) +  \widehat{\alpha}_n \widehat{\beta}_n \widehat{\gamma}_n(\mathcal{V}_m(0,2,2) + \mathcal{U}_m(0,1)) +  \widehat{\beta}_n^3(\mathcal{U}_m(2,0)-\mathcal{W}_m)\right)\bigg),
\end{align*}
which, by Lemma \ref{Ssum}, Lemma \ref{uvxw} and  Lemma \ref{thmDIFF} has the explicit solution
\begin{align}\label{ehatten}
\widehat{\mathcal{E}}_m &=  \widehat{\alpha}_n^{m-4}\bigg( \widehat{\alpha}_n^3 \widehat{\varepsilon}_n\mathcal{Z}_m(0,0,1)+ \widehat{\alpha}_n^2 \widehat{\beta}_n \widehat{\delta}_n(\widehat{\mathcal{R}}_m(0)+\widehat{\mathcal{X}}_m(0,1)) +  \widehat{\beta}_n^4(\widehat{\mathcal{U}}_m(2,0)-\widehat{\mathcal{W}}_m)\\\notag
&+  \widehat{\alpha}_n \widehat{\beta}_n^2 \widehat{\gamma}_n(\widehat{\mathcal{S}}_m(2,0,0)-\widehat{\mathcal{T}}_m(0)+\widehat{\mathcal{V}}_m(0,2,2)+\widehat{\mathcal{U}}_m(0,1))+ \widehat{\alpha}_n^2 \widehat{\gamma}_n^2\widehat{\mathcal{S}}_m(0,1,0)\bigg).
\end{align}
It follows from (\ref{AAA}) and (\ref{BBB}) and Lemma \ref{sumident} that $\widetilde{\widehat{\mathcal{A}}}_p = \widetilde{\widehat{\mathcal{B}}}_p = 0$. We also recall that \[ \widehat{\alpha}_n = x_1^{p^n-2d}\widehat{\varphi},\quad  \widehat{\beta}_n = \frac{x_2}{x_1} \widehat{\alpha}_n,\quad  \widehat{\gamma}_n = -\left(\frac{3x_1}{2}-\frac{x_3}{x_1}\right) \widehat{\alpha}_n.\] 
Concerning (\ref{chatten}) together with (\ref{betaalphagamma}) and %, (\ref{dhatten}) and (\ref{ehatten})
 by letting $m=p$ we obtain \begin{align}\label{ccc}
\widetilde{\widehat{\mathcal{C}}}_p &= \widehat{\alpha}_n^{p-2} \widehat{\beta}_n^2\left(\widetilde{\mathcal{S}}_p(2,0)-\widetilde{\mathcal{T}}_p\right) +  \widehat{\alpha}_n^{p-1} \widehat{\gamma}_n\widetilde{\mathcal{S}}_p(0,1)\\
&=  \widehat{\alpha}_n^{p-2}( \widehat{\beta}_n^2- \widehat{\alpha}_n \widehat{\gamma}_n)\notag\\
%&= \widehat{\alpha}_n^{p-2}\left(\left(\frac{x_2}{x_1} \widehat{\alpha}_n\right)^2 +\left(\frac{3x_1}{2}-\frac{x_3}{x_1}\right)  \widehat{\alpha}_n^2\right)\notag\\
%&=  \widehat{\alpha}_n^p x_1^{-2}\left(x_2^2 +\frac{3x_1^3}{2}-x_1x_3\right)\notag\\
&= \widehat{\alpha}_n^p x_1^{-2}\widehat{\varphi}=(x_1^{p^n-2d}\widehat{\varphi}^d)^px_1^{-2}\widehat{\varphi}\notag\\
&=x_1^{p^{n+1}-2(dp+1)}\widehat{\varphi}^{dp+1}=\widehat{\alpha}_{n+1}.\notag
\end{align}
Similarly for (\ref{dhatten}) we obtain \begin{align*} \widetilde{\widehat{\mathcal{D}}}_p &=  \widehat{\alpha}_n^{p-3}( \widehat{\alpha}_n^2 \widehat{\delta}_n\widetilde{\mathcal{X}}_p(0,1) +  \widehat{\alpha}_n \widehat{\beta}_n \widehat{\gamma}_n(\widetilde{\mathcal{V}}_p(0,2,2) + \widetilde{\mathcal{U}}_p(0,1)) +  \widehat{\beta}_n^2(\widetilde{\mathcal{U}}_p(2,0)-\widetilde{\mathcal{W}}_p))\\
&= \widehat{\alpha}_n^{p-3}( \widehat{\beta}_n^3 -  \widehat{\alpha}_n \widehat{\beta}_n \widehat{\gamma}_n).
\end{align*}
By using $ \widehat{\alpha}_n^{p-2}( \widehat{\beta}_n^2- \widehat{\alpha}_n \widehat{\gamma}_n) = \widehat{\alpha}_{n+1}$ from (\ref{ccc}) we obtain \[\widetilde{\widehat{\mathcal{D}}}_p = \frac{x_2}{x_1}\widehat{\alpha}_{n+1}=\widehat{\beta}_{n+1}.\]
Finally for (\ref{ehatten}) again by using $ \widehat{\alpha}_n^{p-2}( \widehat{\beta}_n^2- \widehat{\alpha}_n \widehat{\gamma}_n) = \widehat{\alpha}_{n+1}$ together with Lemma \ref{Ssum} and \ref{uvxw} %we have
\begin{align*}
\widetilde{\widehat{\mathcal{E}}}_p &=  \widehat{\alpha}_n^{p-4}\bigg( \widehat{\alpha}_n^3 \widehat{\varepsilon}_n\widetilde{\mathcal{Z}}_p(0,0,1)+ \widehat{\alpha}_n^2 \widehat{\beta}_n \widehat{\delta}_n(\widetilde{\widehat{\mathcal{R}}}_p(0)+\widetilde{\widehat{\mathcal{X}}}_p(0,1)) +  \widehat{\beta}_n^4(\widetilde{\widehat{\mathcal{U}}}_p(2,0)-\widetilde{\widehat{\mathcal{W}}}_p)\\
&\hspace{5mm}+  \widehat{\alpha}_n \widehat{\beta}_n^2 \widehat{\gamma}_n(\widetilde{\widehat{\mathcal{S}}}_p(2,0,0)-\widetilde{\widehat{\mathcal{T}}}_p(0)+\widetilde{\widehat{\mathcal{V}}}_p(0,2,2)+\widetilde{\widehat{\mathcal{U}}}_p(0,1))+ \widehat{\alpha}_n^2 \widehat{\gamma}_n^2\widetilde{\widehat{\mathcal{S}}}_p(0,1,0)\bigg)\\
&= \widehat{\alpha}_n^{p-4}( \widehat{\alpha}_n \widehat{\beta}_n^2 \widehat{\gamma}_n -  \widehat{\alpha}_n^2 \widehat{\gamma}_n^2)\\
&=-\left(\frac{3x_1}{2}-\frac{x_3}{x_1}\right)\widehat{\alpha}_{n+1} = \widehat{\gamma}_{n+1}.
\end{align*}
Thus, the proof is finished by specializing for each $i \in [1,3]$ $x_i$ to $a_{i+1}$ which yields 
\begin{align*}\widetilde{\Delta}_p(\zeta) 
%&
\equiv %f^{p^{n+1}}(\zeta)-\zeta\\&\equiv
\alpha_{n+1}\zeta^{2(dp+1)+1} + \beta_{n+1}\zeta^{2(dp+1)+2} + \gamma_{n+1}\zeta^{2(dp+1)+3} \mod \langle \zeta^{2(dp+1)+4}\rangle.
\end{align*}
This completes the proof of Theorem \ref{pncoeff}.
%\begin{example}
%For a power series $f(\zeta) = 7\zeta^2 + 3\zeta^3 + \text{\emph{(higher order terms)}}$, we have that $\coeff(f)=7$.
%\end{example}
\end{proof}

\appendix
\section{Details for Remark \ref{appremark}}\label{apprem}
% the \\ insures the section title is centered below the phrase: AppendixA
Let $p\geq5$ and let $h \in k[[\zeta]]$ of the form $h(\zeta) \equiv \zeta(1+ x_3\zeta^3) \mod \langle \zeta^{10} \rangle$, then by putting $q:=3$, $a_1:=\frac{x_3}{3}$ and $a_2:=-\frac{4}{9}a_3^2$ in (3.2) in \cite[Main Lemma]{LindahlRiveraLetelier2015} we have $i_n(h) = 3(1+p+\dots+p^n)$.

We recall that $0<|a_2|<1$ and $|a_3|=1$, thus we have $\widetilde{f}(\zeta) = \zeta(1 + a_3\zeta^3) \mod \langle \zeta^{10} \rangle$, and  by the previous argument we have $\widetilde{f}$ is 3-ramified as required.%, i.e. the conditions of Corollary \ref{corrperpoints} are met.

\section{Details for Example \ref{appexample}}\label{appexp}
% the \\ insures the section title is centered below the phrase: Appendix B
Let $q_1(\zeta) = \zeta + (1+t)\zeta^3  +\zeta^4$, and $q_2(\zeta) = \zeta + (2+t)\zeta^3 + 4\zeta^4 +4\zeta^5$. In both cases $q_1$ and $q_2$ we have $\lambda \neq 0$. Thus, both $q_1$ and $q_2$ are 2-ramified. 
However, for the reduction we first note that $i_0(\widetilde{q_1}) = i_0(\widetilde{q_2}) = 2$. 
It follows that neither $\widetilde{q_1}$ nor $\widetilde{q_2}$ is 3-ramified. However, note that 
\[\frac{3}{2}1^3 + 1^2 \equiv 0 \pmod{5}, \quad \text{and}\quad \frac{3}{2}2^3+4^2-2\cdot4 \equiv 0 \pmod{5},\] 
so neither $\widetilde{q_1}$ nor $\widetilde{q_2}$ is 2-ramified. 
In fact $i_1(\widetilde{q}_1) = 17$ and $i_1(\widetilde{q}_2) = 27$. 
%Hence
By \cite[Corollary 1]{LaubieSaine1998} %this implies 
\[
i_n(\widetilde{q_1}) = 2 + 3p +\dots+3p^n,
\quad \text{and} \quad i_n(\widetilde{q_2}) = 2 + 5p +\dots+5p^n.
\] 
Thus, $i_n(\widetilde{q_1}) - i_{n-1}(\widetilde{q_1}) = 3p^n$, and by Lemma \ref{lem24} the norm of the periodic points in $\F_5((t))$ of minimal period $p^n$, with $n\geq 1$, are in the case of $q_1$ equal to $|\lambda|^{\frac{1}{p}}$. For $q_2$ we have $i_n(\widetilde{q_2}) - i_{n-1}(\widetilde{q_2}) = 5p^n$. The periodic points of $q_2$, that are not fixed points, are thus in $\{\zeta \in \F_5((t)) : |\zeta| > |\lambda|^{\frac{1}{p}}\}$.

\bibliographystyle{alpha} 
%\bibliography{bib}

\end{document}